\documentclass[11pt]{article}
\usepackage{amsmath} \usepackage{amssymb} \usepackage{latexsym} \usepackage{enumitem} \usepackage{mathrsfs} \usepackage{comment}
\usepackage{color}
\usepackage{tikz}

\usepackage[colorlinks=true,urlcolor=blue,
citecolor=red,linkcolor=blue,linktocpage,pdfpagelabels,bookmarksnumbered,bookmarksopen]{hyperref}

\newtheorem{assumption}{Assumption}

\def\qed{ \ \vrule width.2cm height.2cm depth0cm\smallskip}
\newenvironment{proof}{\noindent {\bf Proof.\/}}{$\qed$\vskip 0.1in}

\newcommand{\ol}{\overline}
\newcommand{\ul}{\underline}

\newcommand{\ba}{\begin{array}}
\newcommand{\ea}{\end{array}}
\newcommand{\be}{\begin{equation}}
\newcommand{\ee}{\end{equation}}
\newcommand{\bea}{\begin{eqnarray}}
\newcommand{\eea}{\end{eqnarray}}
\newcommand{\beaa}{\begin{eqnarray*}}
\newcommand{\eeaa}{\end{eqnarray*}}

\def\dbE{\mathbb{E}}
\def\dbF{\mathbb{F}}

\def\dbL{\mathbb{L}}

\def\dbP{\mathbb{P}}
\def\dbR{\mathbb{R}}

\def\dbX{\mathbb{X}}

\def\a{\alpha}
\def\b{\beta}
\def\g{\gamma}
\def\d{\delta}
\def\e{\varepsilon}

\def\k{\kappa}
\def\l{\lambda}

\def\si{\sigma}

\def\f{\varphi}
\def\th{\theta}

\def\L{\Lambda}

\def\O{\Omega}
\def\cA{{\cal A}}

\def\cE{{\cal E}}
\def\cF{{\cal F}}
\def\cG{{\cal G}}
\def\cH{{\cal H}}

\def\cL{{\cal L}}

\def\cP{{\cal P}}

\def\cS{{\cal S}}

\def\cX{{\cal X}}

\def\no{\noindent}

\def\ms{\medskip}

\def\q{\quad}
\def\qq{\qquad}

\def\pa{\partial}
\def\cd{\cdot}
\def\cds{\cdots}

\def\qed{ \hfill \vrule width.25cm height.25cm depth0cm\smallskip}

\newcommand{\basa}{\begin{assumption}}
\newcommand{\easa}{\end{assumption}}

\newcommand{\bas}{\begin{assum}}
\newcommand{\eas}{\end{assum}}

\def\liminf{\mathop{\underline{\rm lim}}}

\def\pa{\partial}

 \def\cd{\cdot}
\def\cds{\cdots}

\def\dis{\displaystyle}

\def\bx{{\bf x}}
\def\bX{{\bf X}}

\def\1{{\bf 1}}
\def\by{{\bf y}}

\def\:{\!:\!}
\def\reff{\eqref}
\def \proof{{\noindent \bf Proof.\quad}}

at 9pt

\definecolor{alp}{rgb}{0.0, 0.5, 0.0}

\newtheorem{thm}{Theorem}[section]
\newtheorem{lem}[thm]{Lemma}

\newtheorem{prop}[thm]{Proposition}
\newtheorem{rem}[thm]{Remark}
\newtheorem{eg}[thm]{Example}
\newtheorem{defn}[thm]{Definition}
\newtheorem{assum}[thm]{Assumption}

\begin{document}

\title{\bf Stackelberg Games with a Robust Leader} 
\author{Juan Li\thanks{School of Mathematics and Statistics, Shandong University, Weihai, Weihai 264209, P.R. China. Email: juanli@sdu.edu.cn.}, ~Jianfeng Zhang\thanks{\noindent  Department of Mathematics, 
University of Southern California, Los Angeles, CA 90089, USA. E-mail: jianfenz@usc.edu.}, ~ and Zimu Zhu\thanks{\noindent
Fintech Thrust, Hong Kong University of
Science and Technology (Guangzhou), Guangzhou, Guangdong Province, 511453, China. Email: zimuzhu@hkust-gz.edu.cn.
}\ms\\
{\it Dedicated to Professor Shige Peng on the Occasion of His 80th Birthday}
}

\date{}
\maketitle

\begin{abstract} In this paper we study a Stackelberg game with one leader and multiple followers. Given the leader's control, the followers solve a Nash game with possibly multiple equilibria. We consider a robust leader who considers the worst scenario, { which roughly means that the followers would select an equilibrium worst in certain asymptotic sense for the leader.} By using the weak formulation, the problem induces a zero sum game problem with open loop controls, which is time inconsistent. We shall explore heuristically the last problem by using an HJB equation on the Wasserstein space of probability measures, while its rigorous treatment is more involved and is left for future research.The principal-agent problem with one principal and multiple agents can be viewed as a special case of our problem, but with certain constraints.
\end{abstract}

\no{\bf Keywords.}  Stackelberg games, principal-agent problems, Nash equilibrium, zero sum games, open loop controls, mean field control, Wasserstein space

\ms
\no{\it 2020 AMS Mathematics subject classification:}  91A18, 91A15, 91B43, 49L20, 35Q89

\vfill\eject


\section{Introduction}
\label{sect-Introduction}
\setcounter{equation}{0}

Consider a Stackelberg game with one leader: Player $0$, and $n$ followers with $n>1$: Player $i$, $i=1,\cds, n$. For $i=0,\cds, n$, Player $i$ has admissible control $\a^i$ and has utility 
\bea
\label{hata}
J_i(\hat\a), \q i=0, \cds, n,\q\mbox{where}\q \a := (\a^1, \cds, \a^n),~ \hat\a := (\a^0, \a).
\eea
Given $\a^0$, we say $ \a^*$ is a Nash equilibrium with $\a^0$, denoted as $\a^* \in \cE(\a^0)$, if
\bea
\label{Nash0}
J_i(\a^0, \a^*) = v_i(\a^0, \a^{*,-i}), ~i=1,\cds, n,\q\mbox{where}\q v_i(\a^0, \a^{*,-i}) := \sup_{\a^i} J_i(\a^0, \a^{*,-i}, \a^i).
\eea
Here $\a^{-i} := (\a^1,\cds, \a^{i-1}, \a^{i+1}, \cds, \a^n)$.  In the standard literature, the followers' problem is to find the equilibria set $\cE(\a^0)$ for any given $\a^0$, and the leader's problem is:
\bea
\label{leader-sup}
\ol v_0 := \sup_{\a^0} \sup_{  \a^*\in \cE(\a^0)} J_0(\a^0,  \a^*).
\eea
The principal agent problem with one principal and $n$ agents is a special case of the above problem, but with the individual rationality constraints. That is, the above supremum is taken over all $\a^0$ and $\a^*\in \cE(\a^0)$ satisfying the IR constraints: for some $R=(R_1, \cds, R_n)\in \dbR^n$, 
\bea
\label{PA0}
J_i(\a^0,  \a^*) \ge R_i, \q i=1,\cds, n.
\eea
Such problems have received strong attention in the literature, see, e.g., Dempe \cite{Dempe-book}, Dempe-Dutta-Mordukhovich \cite{DDM},  Djete \cite{Djete}, Elie-Mastrolia-Possamai \cite{EMP}, Elie-Possamai \cite{EP},   Green-Stokey \cite{GS}, Koo-Shim-Sung \cite{KSS}, Mookherjee \cite{Mookherjee},   Shimizu-Ishizuka-Bard \cite{SIB}, to mention a few.  We also refer to the textbook Fudenberg-Tirole \cite{FT-book} for general game theory.

There is a fundamental concern in the above formulation though. By considering the optimization $\sup_{\a^*\in \cE(\a^0)}J_0(\a^0, \a^*)$, it authorizes the leader to select an equilibrium, best for her interest. However, equilibria selection is a very serious issue in game theory, see, e.g., Harsanyi-Selten \cite{HS-book}. The one best for the leader may even not be Pareto optimal for the followers,\footnote{\label{wishful} This is not an issue when there is only one follower. Indeed, in this case $\a^*$ becomes the follower's optimal control. When there are multiple optimal controls, they induce the same value for the follower, then the follower is indifferent on them and thus it could be reasonable to assume that the follower would do the leader a favor and choose the optimal control best for the leader. }
then in many practical situations the leader should not expect the followers would implement her favorist equilibrium. Indeed, as we see in Example \ref{eg-supsup} below, such a wishful thinking could lead to a disaster for the leader when the followers do not follow her choice. To overcome this, in this paper we consider a robust  or pessimistic leader, by replacing $\sup_{  \a^*\in \cE(\a^0)}$ with $\inf_{  \a^*\in \cE(\a^0)}$:
\bea
\label{leader-inf}
\ul v_0 := \sup_{\a^0} \inf_{\a^*\in \cE(\a^0)} J_0(\a^0,  \a^*).
\eea
That is, since the leader is not sure which equilibrium the followers may implement, she considers the worst scenario and chooses her own control $\a^0$ accordingly. This problem is much harder to analyze, and there have been serious efforts on static models, see, e.g.,  Aussel-Svensson \cite{AS}, Dassanayaka \cite{Dassanayaka}, Dempe-Mordukhovich-Zemkoho \cite{DMZ},  Liu-Fan-Chen-Zheng \cite{LFCZ1, LFCZ2}, and the references therein. { The work closest to ours is Guo-Hu-Zhang \cite{GHZ}\footnote{ We are grateful for an anonymous referee who brought to our attention this interesting work, which we were not aware of when preparing for the first version of this manuscript.}, which studies a Stackelberg mean field game with a robust leader in a discrete time setting. However, the focus of \cite{GHZ} is different from ours and we shall provide more detailed comments in the end of Introduction.} We also refer to Huang-Wang-Wu \cite{HWW} for a highly relevant work concerning robust Stackelberg game with model uncertainty in a linear quadratic framework. However, to the best of our knowledge, the theory { concerning a robust leader with multiple followers} remains largely open for general continuous time stochastic models.\footnote{If the leader has further information on the followers and has a prediction on which equilibrium the followers may select, then we shall replace $\inf_{  \a^*\in \cE(\a^0)}$ with that prediction and formulate the leader's problem accordingly.   We refer to Mookherjee \cite[Section 3]{Mookherjee} and Zhang \cite[Section 7.1]{Zhang-efficiency} for some preliminary discussions in this direction. }

The problem \reff{leader-inf} is inconvenient to study also because it deals with true equilibria of the followers. The set $\cE(\a^0)$ is very sensitive on $\a^0$, and in practice quite often the leader may use approximate optimal control $\a^{\e,0}$, instead of the optimal control $\a^{*0}$. Since $\cE(\a^{\e,0})$ can be very different from $\cE(\a^{*0})$, then the corresponding value could be far away from $\ul v_0$. This may invalidate many numerical methods for the leader's problem. Moreover, in many practical situations the existence of true equilibrium can also be a serious issue, namely $\cE(\a^0)$ could be empty. To overcome these difficulties, we shall consider the asymptotic value of approximate equilibria, which corresponds to the set value of games introduced in Feinstein-Rudloff-Zhang \cite{FRZ}, instead of the raw set value { there}.

To be precise, we reformulate the robust leader's problem \reff{leader-inf} as follows. First, given $\a^0$ and $\e>0$, we say $  \a^\e$ is an $\e$-equilibrium with $\a^0$, denoted as $  \a^\e \in \cE_\e(\a^0)$, if
\bea
\label{Nashe}
J_i(\a^0,   \a^\e) \ge v_i(\a^0,   \a^{\e,-i})-\e, ~i=1,\cds, n,
\eea
where $v_i$ is defined in \reff{Nash0}. We then define the leader's value as\footnote{ Note that $V_0$ itself is not defined as an optimal value, so we do not talk about optimal controls for $V_0$. For each $\e>0$, $V_\e$ may have $\e$-optimal controls (or optimal controls under appropriate conditions) $\a^{0,\e}$. By definition $v_\e(\a^{0,\e})$ converges to $V_0$, but in general $\a^{0,\e}$ does not converge.}:
\bea
\label{leader}
V_0 := \lim_{\e\to 0} V_\e,\q\mbox{where}\q V_\e := \sup_{\a^0} v_\e(\a^0),\q v_\e(\a^0) := \inf_{ \a^\e\in \cE_\e(\a^0)} J_0(\a^0,   \a^\e).
\eea
Note that { $\cE_\e(\a^0)$ is decreasing  as $\e\downarrow 0$, then $v_\e$ and $V_\e$} are increasing, and thus $V_0$ is well defined, under mild conditions. { Since $\cE_\e(\a^0) \supset \cE(\a^0)$, it is clear that $V_\e \le \ul v_0$ for all $\e>0$ and hence $V_0 \le \ul v_0$.  We note again that  it is is possible $V_0 < \ul v_0$,} see, e.g., Example \ref{eg-Nashe} below.

Our main goal of this paper is to solve the problem \reff{leader} in a diffusion model with closed loop drift controls. As usual in the literature, we use the weak formulation. That is, we fix the state process $X$ as the canonical process, and the players control the law of the state process.  We first study the problem $v_\e(\a^0)$ for a given $\a^0$. Clearly the key is to understand the set $\cE_\e(\a^0)$, and we shall use a characterization of $\e$-equilibria from \cite{FRZ} through certain system of BSDEs, with its solution denoted as $(\ol Y^{\a^0, \a^\e}, \ol Z^{\a^0,\a^\e})$. Next, instead of considering BSDEs with control $\a^\e$, we rewrite $\ol Y$ as a forward SDE with controls $y:=\ol Y_0$ and $\b:= \ol Z$, which satisfy certain constraints. Moreover, thanks to the relaxation to approximate equilibria in \reff{leader}, we can use penalization to get rid of the constraints, and thus, roughly speaking, \reff{leader} becomes:
\bea
\label{leader2}
V_0 \approx \tilde V_0:= \sup_{\a^0} \inf_{y, \b} \tilde J_0(\a^0, y, \b)
\eea
for some appropriate $\tilde J_0$. We remark that, the idea of transforming a BSDE into a forward diffusion has been used in various areas, see, e.g., Ma-Yong \cite{MY1}  for FBSDEs,  Sannikov \cite{Sannikov}  and Cvitanic-Possamai-Touzi \cite{CPT1, CPT2} for principal agent problems, and E-Han-Jentzen \cite{EHJ} for deep learning algorithms for nonlinear PDEs and BSDEs. 

The optimization $\inf_y$ seems simple, since $y\in \dbR^n$ is finite dimensional.  However, note that $\tilde J_0(\a^0, y, \b)$ involves expectation of certain auxiliary state process $\bX$, which is different from the original state process $X$, and thus, given $(\a^0, \b)$, naturally the deterministic optimal $y^*$ involves the distribution of $\bX$. That is,  
\bea
\label{leader3}
\tilde V_0 = \sup_{\a^0} \inf_{\b} \hat J_0(\a^0, \b),\q\mbox{where}\q \hat J_0(\a^0,\b) := \inf_y \tilde J_0(\a^0, y, \b).
\eea
Note that $\hat J_0(\a^0,\b)$ is of mean field type, or say shares the feature of the mean variance problem, which is well known to be time inconsistent in the standard sense. To overcome this difficulty, it is natural to raise the problem to the Wasserstein space of probability measures, namely to view $\tilde V_0$ as a function of $\cL_{\bX}$, the law of the state process $\bX$, rather than a function of the value of $\bX$. For the mean field theory, we refer to the books Carmona-Delarue \cite{CD1,CD2}. 

The problem \reff{leader3} is in the form of zero sum game problem, but with open loop controls.\footnote{Our original problem \reff{leader} uses closed loop controls. However, since we use weak formulation, then the controls $\a^0, \b$ here depend on the canonical process $X$, and thus \reff{leader3} has open loop controls.}  This problem is still time inconsistent, 
and even under the Isaacs condition, typically the $\sup\inf$ and $\inf\sup$ values are  different and there is no saddle point.  See Example \ref{eg-zerosum} below for a counterexample. In particular, its value does not coincide with the solution to the corresponding HJB-Isaacs equation. The latter has been a powerful tool for zero sum games, either with closed loop controls, see, e.g., Hamadene-Lepeltier \cite{HL}, Hamadene-Lepeltier-Peng \cite{HLP}, Cardaliaguet-Rainer \cite{CR1},  Pham-Zhang \cite{PZ}, Sirbu \cite{Sirbu1, Sirbu2}, or with strategy versus open loop controls, see, e.g., Fleming-Souganidis \cite{FS}, Buckdahn-Li \cite{BL}, or with delayed strategies, see, e.g., Cardaliaguet-Rainer \cite{CR2}. We refer to Possamai-Touzi-Zhang \cite{PTZ} for detailed discussions on various formulations and for more references, and we also refer to  Cosso-Pham \cite{CP} for mean field type zero sum games with strategy versus open loop controls. For open loop controls, the literature has been mainly on discrete models, see, e.g., Basar-Olsder \cite{BO-book} and the references therein. To the best of our knowledge, the theory remains open for diffusion models.\footnote{We would like to mention that there are quite a few works on linear quadratic diffusion models, see, e.g., Feng-Hu-Huang \cite{FHH}, Mou-Yong \cite{MY2}, Sun \cite{Sun}, and Sun-Wang-Wen \cite{SWW}. However, their main focus is the saddle points for the zero sum game, which does not exist in general, see Remark \ref{rem-saddle} (ii) below.}

Our second goal of the paper is to study general zero sum games with open loop controls in diffusion models with possibly mean field (or McKean-Vlasov) dynamics. Our focus is the characterization of the value $\tilde V_0$ in \reff{leader3}, which is always well defined, rather than the saddle point, which typically does not exist.  Notice that this problem can also be viewed as a Stackelberg game, but with only one follower, and thus some of our previous ideas can be applied to this problem as well. We emphasize that, since there is only one follower and due to the zero sum nature, \reff{leader-sup} and \reff{leader-inf} are the same for this Stackelberg game. This simplifies the problem a lot. { On the other hand, however,} since we are using open loop controls now, even when there is only drift control, we are not able to fix the state process and use the weak formulation. Consequently, we are not able to characterize the follower's problem through an BSDE. Instead, we characterize the follower's problem $\inf_\b \hat J_0(\a^0, \b)$ through a coupled FBSDE, which is significantly harder to analyze than BSDEs. Nevertheless, by assuming strong technical conditions such that the above involved FBSDEs are well-posed, we manage to (approximately) characterize the problem $\tilde V_0$ in \reff{leader3} through an HJB equation on the Wasserstein space. In particular, for standard zero sum game with open loop controls but without the law involvement, this infinitely dimensional HJB equation can be reduced to a finite dimensional BSPDEs. { We shall point out that the analysis here is heuristic. The imposed technical assumptions are rather strong and hard to verify, in particular they are violated when \reff{leader3} is derived from our original Stackelberg game problem \reff{leader}. So a complete solution to  \reff{leader} remains open and we shall investigate it in future research.}

Finally, our approach remains effective\footnote{ But the last part is also heuristic, as mentioned in the previous paragraph.} for principal agent problem with multiple agents. In particular, the IR constraint \reff{PA0} will only affect the optimization $\inf_y$ in \reff{leader2}, which makes the $\hat J_0(\a^0, \b)$ in \reff{leader3} relies on $\cL_\bX$ in a slightly more involved but still semi-explicit way. Moreover, we may allow for lump sum payments as well, which will add another layer of optimization for the terminal condition of the above infinite dimensional HJB equation characterizing $\tilde V_0$, but will have no impact on the generator of this HJB equation.  

{ We conclude the Introduction with a brief comparison with the very interesting work \cite{GHZ}. Roughly speaking, in its contexts, \cite{GHZ} considers the problem $V_\e$ in \reff{leader} for fixed $\e$, justified by the bounded rationality. It characterizes $V_\e$ by another constrained optimization problem with constraints depending on $\e$, and shows that the value is unstable with respect to model perturbations and is right continuous in $\e$ (but in general not left continuous). We instead consider the limit value $V_0:= \lim_{\e\to 0} V_\e$ in \reff{leader}, which is independent of $\e$ and does not require the bounded rationality. We remark that our analyses in this paper, e.g. Theorem \ref{thm-Vn}, are valid only for the limit value $V_0$, but not for $V_\e$. Moreover, while not investigated in this paper, as we see in \cite{FRZ}, nice properties such as regularity/stability and dynamic programming principle are more likely to hold for the limit value. 
}

The rest of the paper is organized as follows. In Section \ref{sect-stackelberg} we introduce our model and transform the Stackelberg game problem into a zero sum game problem with open loop controls. In Section \ref{sect-inconsistency} we discuss the time inconsistency issue. In Section \ref{sect-PA} we present a principal agent problem. In Section \ref{sect-drift} we analyze a special case of the  zero sum game problem with open loop controls, and in Section \ref{sect-vol} we study the general case. In Subsections \ref{sect-app1} and \ref{sect-app2} we { formally} apply the above general result to the Stackelberg game and to the principal agent problem, respectively, { by ignoring the required technical conditions}. Finally in Section \ref{sect-eg} we present a few relevant examples.

\section{A continuous time Stackelberg game}
\label{sect-stackelberg}
\setcounter{equation}{0} 

Fix a finite time horizon $[0, T]$. Let $(\O, \cF, \dbP)$ be a probability space, $B$ a standard Brownian motion, and $\dbF := \dbF^B$. Let $A\subset \dbR$ be a domain for the control values, and $\cA$ the set of $\dbF^B$-progressively measurable $A$-valued processes. For notational simplicity, we assume all processes are one dimensional, but all the results in the paper can be easily extended to multiple dimensions. We shall use weak formulation. Set 
\bea
\label{X}
X \equiv B.
\eea
For $\hat\a:= (\a^0,   \a)\in \cA^{n+1}$, denote 
\bea
\label{Girsanov}
\left.\ba{c}
\dis B^{\hat\a}_t := B_t - \int_0^t b(s, X_s, \hat \a_s) ds,\\
\dis\mbox{where}\q  {d\dbP^{\hat \a}\over d\dbP} := M^{\hat\a}_T := \exp\Big(\int_0^T b(s, X_s, \hat \a_s) dB_s - {1\over 2} \int_0^T |b(s, X_s, \hat \a_s)|^2 ds\Big).
\ea\right.
\eea
For $i=0,\cds, n$,
\bea
\label{Ji}
J_i(\hat\a) := \dbE^{\dbP^{\hat \a}}\Big[g_i(X_T) + \int_0^T f_i(s, X_s, \hat\a_s) ds\Big].
\eea
It is well known that $J_i(\hat \a) = Y^{\hat\a, i}_0$, where
\bea
\label{Yai}
\left.\ba{c}
\dis Y^{\hat\a, i}_t = g_i(X_T) + \int_t^T F_i(s, X_s, \hat \a_s, Z^{\hat\a, i}_s) ds - \int_t^T Z^{\hat\a, i}_sdB_s, \\
\dis \mbox{where}\q F_i(t, x, \hat a, z_i) := f_i(t, x, \hat a) + b(t, x, \hat a)z_i.
\ea\right.
\eea
Here $a=(a_1, \cds, a_n)\in A^n$  and $\hat a = (a_0, a)\in A^{n+1}$. Introduce further that, for $i=1,\cds, n$,
\bea
\label{barYai}
\left.\ba{c}
\dis \ol Y^{\hat\a, i}_t = g_i(X_T) + \int_t^T \ol F_i(s, X_s, \hat \a_s, \ol Z^{\hat\a, i}_s) ds - \int_t^T \ol Z^{\hat\a, i}_sdB_s, \\
\dis \mbox{where}\q \ol F_i(t, x, \hat a, z_i) := \sup_{a_i'\in A} F_i(t, x, a_0, a^{-i}, a_i', z_i).
\ea\right.
\eea
As standard in the literature, for the purpose of studying the leader's problem, we rewrite the above BSDE as a forward equation. That is, for any $\b = (\b^1, \cds, \b^n) \in (\dbL^2(\dbF))^n$, define
\bea
\label{cXi}
\cX^{\hat\a, \b, i}_t :=  -  \int_0^t \ol F_i(s, X_s, \hat \a_s, \b^{i}_s) ds + \int_0^t  \b^{i}_sdB_s
\eea
Then, denoting $y := \ol Y^{\hat\a}_0 := (\ol Y^{\hat\a, 1}_0,\cds, \ol Y^{\hat\a, n}_0)$, $\b := \ol Z^{\hat\a} := (\ol Z^{\hat\a,1},\cds, \ol Z^{\hat\a,n})$, we have
\bea
\label{Y=cX}
 \ol Y^{\hat\a, i}_t = y_i + \cX^{\hat\a, \b, i}_t.
\eea
We note that $\ol F_i$ depends on $\hat a$ only through $(a_0, a^{-i})$, and $\cX^{\hat\a, \b, i}$ depends on $\b$ only through $\b^i$. 

 For technical simplicity, we impose the following strong conditions. We note that these conditions, especially the boundedness requirements, can be weakened significantly.

\begin{assum}
\label{assum-0}
The coefficients $b, f_i, g_i$ are bounded and uniformly continuous in all variables.  
\end{assum}
Clearly, under the above assumption, Girsanov theorem holds true for \reff{Girsanov}, and $\ol F$ is Lipschitz continuous in $z$ and thus BSDE \reff{barYai} is well-posed.

We first characterize the true Nash equilibria. The following result is obvious.
\begin{lem}
\label{lem-Nash}
Let Assumption \ref{assum-0} hold and fix $\a^0\in \cA$. Then $\a^*\in \cE(\a^0)$ if and only if 
\bea
\label{NEcharacterization1}
\ol F_i(s, X_s, \a^0_s,   \a^*_s, \ol Z^{(\a^0,   \a^*), i}_s) =  F_i(s, X_s, \a^0_s,   \a^*_s, \ol Z^{(\a^0,   \a^*), i}_s), ~ds\times d\dbP-a.s., i=1,\cds, n,
\eea
and if and only if there exist $y\in \dbR^n$ and $\b\in (\dbL^2(\dbF))^n$ such that
\bea
\label{NEcharacterization2}
\left.\ba{c}
\dis \ol F_i(s, X_s, \a^0_s,   \a^*_s, \b^{ i}_s) =  F_i(s, X_s, \a^0_s,   \a^*_s, \b^{i}_s), ~ds\times d\dbP-a.s.,\ms\\
\dis y_i + \cX^{(\a^0,   \a^*), \b, i}_T = g_i(X_T),~ a.s.,
\ea\right.~ i=1,\cds, n.
\eea
\end{lem}

We next investigate approximate Nash equilibria. Denote
\bea
\label{Ierror}
\left.\ba{c}
\dis I(\hat\a, y,\b) := \sum_{i=1}^n \dbE\Big[ \int_0^T |\ol F_i(s, X_s, \hat \a_s, \b^{ i}_s) - F_i(s, X_s, \hat\a_s, \b^{i}_s)|^{3\over 2} ds \\
\dis + \big|y_i + \cX^{\hat\a, \b, i}_T - g_i(X_T)\big|^2\Big]. 
\ea\right.
\eea
The following result follows from \cite[Theorem 4]{FRZ} and plays a crucial role in this paper.
\begin{prop}
\label{prop-Nashe}
Assume Assumption \ref{assum-0} holds and fix $\a^0\in \cA$ and $\e>0$. Then there exists $\d>0$ such that the following hold.

\no(i) For any $  \a^\d\in\cE_\d(\a^0)$,  $I(\a^0,  \a^\d,  J(\a^0, \a^\d), \ol Z^{\a^0, \a^\d}) \le \e$, where $J := (J_1, \cds, J_n)$.

\no(ii) For any $  \a\in \cA^n$, if there exist $y$, $\b$ such that $I(\a^0,  { \a}, y, \b)\le \d$, then $  \a\in \cE_\e(\a^0)$. 
\end{prop}

Recall \reff{leader} and denote
\bea
\label{Vlambda}
V_\l :=  \sup_{\a^0\in \cA} v_\l(\a^0),\q v_\l(\a^0) := \inf_{\a\in \cA^n, y\in \dbR^n, \b\in (\dbL^2(\dbF))^n}  \Big[J_0(\a^0, \a) + \l  I(\a^0, \a, y, \b)\Big]
\eea
We note that, in \reff{leader} the controls $\a^\e$ satisfy the $\e$-equilibrium constraint, while the above optimization problem is unconstrained and thus is a lot easier to solve. 

\begin{thm}
\label{thm-Vn}
Assume Assumption \ref{assum-0} holds and $\cE_\e(\a^0) \neq \emptyset$ for any $\a^0$ and $\e>0$. Then 
\bea
\label{Vlambda-conv}
V_\l \uparrow V_0,\q\mbox{as}~ \l\uparrow \infty.
\eea
\end{thm}
\proof Under Assumption \ref{assum-0}, it is clear that $V_\l$ is finite and increasing in $\l$.  

{\bf Step 1.} Fix $\l$. For any $\e>0$, choose $\a^0$ such that $v_\l(\a^0)\ge V_\l - \e$.  Let $\d>0$ be small enough such that Proposition \ref{prop-Nashe} (i) holds. Choose $\a^\d\in \cE_\d(\a^0)$ such that $J_0(\a^0, \a^\d) \le  v_\d(\a^0) + \e$.  Set $y := J(\a^0, \a^\d)$ and $\b :=  \ol Z^{\a^0, \a^\d}$. By Proposition \ref{prop-Nashe} (i) we have $I(\a^0,  \a^\d,  y, \b) \le \e$. Then
\beaa
V_\l &\le& v_\l(\a^0)+\e \le J_0(\a^0, \a^\d) + \l  I(\a^0, \a^\d, y, \b) +\e \\
&\le& v_\d(\a^0) + \e + \l \e +\e \le V_\d + (\l+2)\e\le V_0+(\l+2)\e.
\eeaa
Since $\e>0$ is arbitrary, we have $V_\l\le V_0$. Moreover, since $\l$ is arbitrary, we obtain $\dis\lim_{\l\to \infty} V_\l \le V_0$.

{\bf Step 2.} Fix $\e>0$. Let $\d>0$ be small enough such that both Proposition \ref{prop-Nashe} (i) and (ii) hold, and choose $\a^{0}$ such that $v_\e(\a^{0}) \ge V_\e -\d$. For any $\l\ge 1$, choose $(\a^\l, y^\l, \b^\l)$ such that 
\bea
\label{vlest1}
J_0(\a^{0}, \a^\l) + \l  I(\a^{0}, \a^\l, y^\l, \b^\l) \le v_\l(\a^0) + {1\over \l}.
\eea
 Setting $\tilde\e := {1\over \l}$ in Proposition \ref{prop-Nashe} (i) and for corresponding $\tilde\d>0$, since  $\cE_{\tilde\d}(\a^0)\neq \emptyset$, there exists $\tilde \a^\l$ such that $I(\a^{0}, \tilde \a^\l, \tilde y^\l, \tilde\b^\l) \le {1\over \l}$, where $\tilde y^\l := J(a^0, \tilde \a^\l)$, $\tilde \b^\l := \ol Z^{\a^0, \tilde \a^\l}$. Then 
\beaa
v_\l(\a^0) \le J_0(\a^0, \tilde \a^\l) + \l  I(\a^{0}, \tilde \a^\l, \tilde y^\l, \tilde\b^\l) \le  J_0(\a^0, \tilde \a^\l) + 1 \le C_0,
\eeaa
for a constant $C_0$ depending on the bounds in Assumption \ref{assum-0}. By \reff{vlest1}, this implies that
\beaa
\l I(\a^0, \a^\l, y^\l, \b^\l) \le C_0 +{1\over \l} - J_0(\a^0, \a^\l) \le C_1,
\eeaa
 for another constant $C_1$ depending on the bounds in Assumption \ref{assum-0}. Then 
 \beaa
 0\le I(\a^0, \a^\l, y^\l, \b^\l) \le {C_1\over \l}\le \d,
 \eeaa
  for $\l$ large enough. By Proposition \ref{prop-Nashe} (ii) we have $\a^\l \in \cE_\e(\a^0)$. Thus
\beaa
V_\e &\le& v_\e(\a^{0}) + \d \le J_0(\a^0, \a^\l) +\d \le J_0(\a^{0}, \a^\l) + \l  I(\a^{0}, \a^\l, y^\l, \b^\l) +\d \\
&\le& v_\l(\a^0) + {1\over \l} +\d \le V_\l + {1\over \l} +\d.
\eeaa
Then $\dis V_\e\le \lim_{\l\to \infty} V_\l + \d$. Since $\d>0$ can be arbitrarily small, we have $\dis V_\e\le \lim_{\l\to \infty} V_\l$. Send $\e\to 0$, we obtain $\dis  V_0\le \lim_{\l\to \infty} V_\l$, and hence equality holds. 
\qed

{
\begin{rem}
\label{rem-cEenonempty}
In Theorem \ref{thm-Vn} we assume $\cE_\e(\a^0) \neq \emptyset$ for all $\e>0$, which is a much weaker requirement than the existence of true equilibrium: $\cE(\a^0) \neq \emptyset$. \\
(i) When $n=1$, namely there is only one follower, the follower's problem reduces to a control problem and $\cE_\e(\a^0) \neq \emptyset$ is always true (as long as $\sup_{\a^1}J_1(\a^0, \a^1)<\infty$). \\
(ii) When $n=2$ and the two followers solve a zero sum game problem: $J_1 + J_2 = 0$, then $\cE_\e(\a^0) \neq \emptyset$ is true as long as the corresponding Isaacs' condition holds true (together with very mild additional technical conditions), while the existence of saddle point (equilibrium) requires stronger conditions. \\
(iii) For the general case, the existence of $\e$-equilibria is still a highly nontrivial issue. There are extensively studies in the literature on various specific games. We refer to \cite[Section 4]{QZ} for some sufficient conditions in the current setting in terms of the set valued Hamiltonian for the static game: $a=(a_1, \cds, a_n) \to F_i(t, x, a_0, a, z_i)$, $i=1,\cds, n$.\end{rem}
 
 }
 
\section{The time inconsistency issue}
\label{sect-inconsistency}
\setcounter{equation}{0} 
We now turn to the problem \reff{Vlambda}, which is an $\sup\inf$ problem without constraint, and thus has the feature of zero sum games. However, the control $y$ is deterministic, which actually complicates the problem. To see this, we note that  in \reff{Ierror} and \reff{Vlambda} $y$ appears only in the term $\dbE\big[|y_i+\cX^{\hat\a, \b, i}_T - g_i(X_T)|^2\big]$. It is clear that
\bea
\label{optimality}
\left.\ba{c}
\dis \inf_{y_i} \dbE\big[|y_i+\cX^{\hat\a,\b, i}_T - g_i(X_T)|^2\big] = Var\big(\cX^{\hat\a,\b, i}_T - g_i(X_T)\big), \ms\\
\dis \mbox{with optimal control}~ y_i^* = \dbE[g_i(X_T) -\cX^{\hat\a,\b, i}_T].
\ea\right.
\eea
Then  \reff{Vlambda}  becomes, recalling \reff{Girsanov} and \reff{Ji},
\bea
\label{Vlambda-2}
\left.\ba{c}
\dis V_\l = \sup_{\a^0\in \cA}\inf_{(\a, \b)\in\cA^n\times (\dbL^2(\dbF))^n} J_\l(\hat\a, \b),\\
\dis \mbox{where}\q J_\l(\hat\a, \b):= \dbE\big[M^{\hat\a}_T g_0(X_T)\big]  + \l \sum_{i=1}^n Var\big(\cX^{\hat\a, \b, i}_T - g_i(X_T)\big)\\
\dis \q +\dbE\Big[\int_0^T \big[M^{\hat\a}_s f_0(s, X_s, \hat\a_s) +\l \sum_{i=1}^n  |\ol F_i(s, X_s, \hat \a_s, \b^{ i}_s) - F_i(s, X_s, \hat\a_s, \b^{i}_s)|^{3\over 2} \big]ds \Big].
\ea\right.
\eea
This problem is time inconsistent in the standard approach for two reasons. First, due to the involvement of the variance term,  it has the same feature as the mean variance problem, which is well known to be time inconsistent and thus the PDE approach fails in the standard sense; Second, although the $\sup\inf$ structure shares the feature of zero sum games, here the controls $\a^0$ and $(\a, \b)$ are in open loop sense, then even without the variance term, the DPP fails and thus we cannot apply the HJB-Issacs equation to characterize the value.  

The mean variance feature, or say the time inconsistency due to the law involvement, can be overcome by  raising the problem to Wasserstein space of probability measures, see e.g. \cite{WZ}. To be precise, consider the $n+2$ dimensional random vector $\bX =(X, \cX^1, \cds, \cX^n, M)$. Given $(t, \mu) \in [0, T]\times \cP_2(\dbR^{n+2})$, and $\eta = (\eta_0,\eta_1, \cds, \eta_n, \eta_{n+1})\in (\dbL^2(\cF_t))^{n+2}$ with $\cL_\eta = \mu$, define
\bea
\label{Vlmu}
\left.\ba{lll}
\dis V_\l(t, \mu) := \sup_{\a^0\in \cA}\inf_{(\a, \b)\in\cA^n\times (\dbL^2(\dbF))^n} \Big\{ \dbE\big[M^{t, \eta, \hat\a}_T g_0(X^{t, \eta}_T)\big] + \l \sum_{i=1}^n Var\big(\cX^{t, \eta, \hat\a, \b, i}_T - g_i(X^{t, \eta}_T)\big)\\
\dis \q + \dbE\Big[\int_t^T \big[ M^{t, \eta,\hat\a}_s f_0(s, X^{t,\eta}_s, \hat\a_s) +\l \sum_{i=1}^n  |\ol F_i(s, X^{t,\eta}_s, \hat \a_s, \b^{ i}_s) - F_i(s, X^{t,\eta}_s, \hat\a_s, \b^{i}_s)|^{3\over 2} \big]ds\Big]  \Big\},
\ea\right.
\eea
where, for $s\in [t, T]$,
\bea
\label{XcXMt}
\left.\ba{lll}
\dis X^{t, \eta}_s = \eta_{0} + B_s-B_t;\\
\dis\cX^{t, \eta, \hat\a, \b, i}_s = \eta_i -  \int_t^s \ol F_i(r, X^{t, \eta}_r, \hat \a_r, \b^{i}_r) dr + \int_t^s  \b^{i}_rdB_r,\q i=1,\cds, n;\\
\dis M^{t,\eta,\hat\a}_s = \eta_{n+1} + \int_t^s M^{t,\eta,\hat\a}_r b(r, X^{t,\eta}_r, \hat\a_r) dB_r.
\ea\right.
\eea

This dynamic value function $V_\l(t, \mu)$, however, is still time inconsistent and thus does not satisfy the naturally associated HJB-Issacs equation on the Wasserstein space. The main reason is that here the controls $\a^0$ and $(\a, \b)$ are all open loop. Before we analyze this in details, let's explain why the $\sup\sup$ problem \reff{leader-sup} is easier. 

\begin{rem}
\label{rem-supsup}
One can easily show that the counterpart of $V_\l$ for \reff{leader-sup} is, by subtracting $I$ for the supremum problem:
\bea
\label{supsup}
V'_\l :=  \sup_{\a^0}  \sup_{\a, y, \b}  \Big[J_0(\a^0, \a) - \l  I(\a^0, \a, y, \b)\Big].
\eea

\no(i) It is clear that  $V'_\l = \sup_y V'_\l(y)$ where
\beaa
\left.\ba{lll}
\dis  V'_\l(y) := \sup_{\a^0, \a,  \b}  \Big[J_0(\a^0,\a) - \l  I(\a^0,\a, y, \b)\Big]\\
\dis\qq\q =  \sup_{\hat\a,  \b}  \dbE\Big[M^{\hat\a}_T g_0(X_T) - \l \sum_{i=1}^n \big|y_i + \cX^{\hat\a, \b, i}_T - g_i(X_T)\big|^2\\
\dis\qq\qq + \int_0^T \big[M^{\hat\a}_s f_0(s, X_s, \hat\a_s)  - \l\sum_{i=1}^n  |\ol F_i(s, X_s, \hat \a_s, \b^{ i}_s) - F_i(s, X_s, \hat\a_s, \b^{i}_s)|^{3\over 2} \big]ds\Big]. 
\ea\right.
\eeaa
For fixed $y$, $V'_\l(y)$ is a standard control problem (with open loop controls) and thus can be solved through a standard HJB equation. Then it is straightforward to solve the finite dimensional optimization problem $V'_\l = \sup_y V_\l'(y)$. We emphasize that the main trick here is that, for $\sup\sup$, one may switch the orders of optimization and thus one can fix $y$ and optimize over the dynamic controls $\a^0, \a, \b$. 

\no(ii) Alternatively,  by \reff{optimality} we can see that
 \bea
\label{Vlambda-2sup}
\left.\ba{c}
\dis V'_\l = \sup_{\a^0}\sup_{\a, \b}J_\l'(\hat\a, \b),\q J_\l'(\hat\a, \b):= \dbE\big[M^{\hat\a}_T g_0(X_T)\big]  - \l \sum_{i=1}^n Var\big(\cX^{\hat\a, \b, i}_T - g_i(X_T)\big)\\
\dis \qq +\int_0^T \big[M^{\hat\a}_s f_0(s, X_s, \hat\a_s) -\l \sum_{i=1}^n  |\ol F_i(s, X_s, \hat \a_s, \b^{ i}_s) - F_i(s, X_s, \hat\a_s, \b^{i}_s)|^{3\over 2} \big]ds \Big] \Big\}.
\ea\right.
\eea
This is a mean field type of control problem. Since for control problems open loop and closed loop controls lead to the same value function, we see that  $V'_\l = V'_\l(0,  \d_{(0,  0, \cds, 0, 1)})$ where the dynamic value function $V'_\l(t, \mu)$ satisfies an HJB equation on the Wasserstein space of probability measures, typically in certain viscosity sense.  For the latter we refer to \cite{ZTZ} and the references therein.
\end{rem}

We now come back to \reff{Vlmu}-\reff{XcXMt} and explain why the open loop controls lead to the time inconsistency. One can easily see that the zero sum game is associated with the following HJB equation on the Wasserstein space: for $\bx = (x_0, x_1, \cds, x_n, x_{n+1})$,
\bea
\label{V1-HJB}
\left.\ba{c}
\dis \pa_t V_1 + H_1(t, \mu, \pa_\mu V_1(t, \mu,\cd), \pa_{\bx\mu} V_1(t, \mu, \cd))=0,\q (t, \mu)\in [0, T]\times \cP_2(\dbR^{n+2}),\\
\dis \mbox{where}\q H_1(t, \mu, \f, \psi) := \int_{\dbR^{n+2}} \sup_{a_0\in A}\inf_{(a,\b)\in A^n\times \dbR^n} h(t, \mu, \bx, \f(\bx), \psi(\bx), \hat a, \b) \mu(d\bx),\\
\dis  h(t, \mu, \bx, z, \g, \hat a,\b) := {1\over 2}\sum_{i=0, 1,\cds, n, n+1} ~\sum_{j=0, 1,\cds, n, n+1}  \g_{ij} \b^i\b^j - \sum_{i= 1,\cds, n} z_i \ol F_i(t, x_0, \hat a, \b^{i})  \\
\dis\q + x_{n+1} f_0(t, x_0, \hat a) +\l \sum_{i=1}^n  |\ol F_i(t, x_0, \hat a, \b^{ i}) - F_i(t, x_0, \hat a, \b^{i})|^{3\over 2},
\ea\right.
\eea
where $\b^0:= 1$, $\b^{n+1} := x_{n+1} b(t,x_0, \hat\a)$. We remark that $H_1$ depends on $\pa_\mu V_1(t,\mu,\cd)$ and $\pa_{\bx\mu} V_1(t,\mu,\cd)$, namely it depends on $V_1$ locally in $(t, \mu)$ but globally in $\bx$. 
Assume the above Hamiltonian has an equilibrium arguments $a_0^* = \phi_0^1(t, \mu, \bx, z,\g)$ and $(a,\b)^* = \phi_1^1(t, \mu, \bx, z,\g)$. Then the optimal control $\a^{0*}$ would take the form $\phi_0\big(\cd, \pa_\mu V_1(\cd) , \pa_{x\mu} V_1(\cd)\big)(t, \cL_{\bX^{\a^*}_t}, \bX^{\a^*}_t)$, which means that $\a^{0*}$ is an equilibrium in the closed-loop form, rather than in the open loop form.  
\begin{rem}
\label{rem-V1}
Under appropriate technical conditions, one can show that the value function for the zero sum game problem \reff{Vlmu} with closed loop controls satisfies  PDE \reff{V1-HJB}. Alternatively, $V_1$ is also the value function for the zero sum game with strategy versus open loop controls: $V_1 = \inf_\cS \sup_{\a^0} J_\l(\a^0, \cS(\a^0))$, where the strategy $\cS$ maps $\a^0$ to $(\a,\b)$ in a non-anticipating way, see \cite{CP}.
\end{rem}

For games, typically open loop and closed loop equilibria are different. Indeed, the $V$ defined by \reff{Vlmu} does not satisfy the PDE \reff{V1-HJB}.  Another choice is, since $\a^0$ is allowed to depend only on $X^0$, for the $h$ in \reff{V1-HJB},
\bea
\label{V2-HJB}
\left.\ba{c}
\dis \pa_t V_2 + H_2(t, \mu, \pa_\mu V_2(t, \mu,\cd), \pa_{\bx\mu} V_2(t, \mu,\cd)) =0, \q  \mbox{where}\\
\dis H_2(t, \mu, \f, \psi) := \int_{\dbR} \sup_{a_0\in A} \ul H_1(t, \mu, x_0, \f, \psi, a_0) \mu_0(dx_0),\\
\dis {\ul H_1(\cds)} := \int_{\dbR^{n+1}}\!\!\inf_{(a,\b)\in A^n\times \dbR^n} h(t, \mu, x_0, \bx^{-0}, \f(x_0, \bx^{-0}), \psi(x_0, \bx^{-0}), a_0, a,\b)\mu(d \bx^{-0} |x_0).
\ea\right.
\eea
Here $\bx^{-0} := (x_1, \cds, x_n, x_{n+1})$, $\mu_0$ is the marginal  distribution $\cL_{X_t}$ of $\mu$, and $\mu(d \bx^{-0} |x_0)$ is the conditional distribution $\cL_{\bX^{-0}_t|X^0_t=x_0}$. Then, the optimal arguments of $H_2$ and $\ul H_1$ take the form 
\bea
\label{phi}
a_0^* = \phi_0(t,\mu, x_0, \f, \psi),\q (a^*,\b^*) = \phi_1(t, \mu, \bx, z, \g, a_0).
\eea
 However, this still does not match \reff{Vlambda-2}. The main reason is that the corresponding $\a^{0*}_t = \phi_0\big(t, \cL_{\bX_t}, X_t, \pa_\mu V_2(t,\cL_{\bX_t},\cd), \pa_{\bx\mu} V_2(t,\cL_{\bX_t},\cd)\big)$ depends on $\a^1$ through $\cL_{\bX_t}$. We note that, here the crucial dependence is not about filtration, indeed, $\bX^{-0}$ itself is $\dbF$-progressively measurable, and thus even $\a^{0*}_t =\phi^1_0\big(\cd, \pa_\mu V_1(\cd) , \pa_{x\mu} V_1(\cd)\big)(t, \cL_{\bX^{\a^*}_t}, \bX^{\a^*}_t)$ corresponding to \reff{V1-HJB} is also $\dbF$-measurable. The key here is that $\a^{0*}$ is not allowed to depend on $(\a,\b)$. That's why DPP holds when we use closed loop controls, where $\a^{0*}$ can react to $(\a,\b)$ through $(\cL_\bX, \bX)$.  However, we note that $(\a^{*}, \b^*)$ is allowed to depend on $\a^0$, and thus  the form $(a^*,\b^*) = \phi_1(t, \mu, \bx, z, \g, a_0)$ does not cause any trouble. 
 
 \begin{rem}
 \label{rem-saddle}
 (i) By introducing a strategy $\cS': \a^0 \to (\a, \b)$, we can easily show that
 \bea
 \label{strategy}
 V_\l=\sup_{\a^0} \inf_{\a, \b} J_\l(\a^0, \a, \b) = \inf_{\cS'} \sup_{\a^0} J_\l(\a^0, \cS(\a^0)).
 \eea
However, we shall note that this strategy $\cS'$ is anticipating, in the sense that $(\cS'(\a^0))_t$ may depend on $\a^0_{[0, T]}$, rather than $\a^0_{[0, t]}$. Then, unlike in \cite{CP}, the right side of \reff{strategy} is still time inconsistent.

\no(ii) Clearly the right side of \reff{strategy} allows more strategy $\cS'$ than the non-anticipating strategy $\cS$ in Remark \ref{rem-V1}, then we see that $V_\l \le V_1$. Similarly, if we define $\ol V_\l$ and $\ol V_1$ corresponding to the problem $\inf_{\a, \b} \sup_{\a^0} J_\l(\a^0, \a, \b)$, then we have $ \ol V_1 \le \ol V_\l$. Therefore,
\beaa
V_\l \le V_1 \le \ol V_1 \le \ol V_\l.
\eeaa
We remark that, under the Isaacs condition, we shall have $V_1 = \ol V_1$. However, even under the Isaacs condition, typically $V_\l < \ol V_\l$, as we see in Example \ref{eg-zerosum} below. In particular, in this case the zero sum game with open loop controls does not have saddle points.
\end{rem}

\section{A continuous time principal-agent problem}
\label{sect-PA}
\setcounter{equation}{0} 
Before we move on to analyze the zero sum game problem \reff{Vlambda-2}, in this section we study a  continuous time principal-agent problem with multiple agents and one robust principal. 
Consider the setting $T$, $(\O, \cF, \dbP)$, $B$, $\dbF=\dbF^B$, $X$, $A$, $\cA$ as in the previous two sections. The agents' controls are still $\a=(\a^1, \cds, \a^n) \in \cA^n$.  For the principal's control, besides $\a^0$,\footnote{The $\a^0$ can be  interpreted as the principal's control on the state process, or as continuous time payments. In the latter case $\a^0 = (\a^{0,1}, \cds, \a^{0,n})$ should be a vector. However, for notational convenience in this section we just consider one dimensional $\a^0$.} as in the majority of the literature, we also consider the lump sum payments: $\xi = (\xi_1, \cds, \xi_n)\in (\dbL^2(\cF_T))^n$. Recall \reff{Girsanov}. Given $\xi$ and $\hat\a$, the agents' and the principal's utilities are
\bea
\label{PA-Ji}
\left.\ba{lll}
\dis J_i(\xi, \hat\a) := \dbE^{\dbP^{\hat\a}}\Big[g_i(\xi_i) + \int_0^T f_i(s, X_s, \hat\a_s) ds\Big] = Y^{\xi, \hat\a, i}_0,~ i=1,\cds, n,\\
\dis  J_0(\xi, \hat\a) := \dbE^{\dbP^{\hat\a}}\Big[g_0(X_T - \sum_{i=1}^n \xi_i) + \int_0^T f_0(s, X_s, \hat\a_s) ds\Big] = Y^{\xi, \hat\a, 0}_0,
\ea\right.
\eea
where, for the $F_i$ in \reff{Yai},
\bea
\label{PA-Yai}
\left.\ba{lll}
\dis Y^{\xi,\hat\a, i}_t = g_i(\xi_i) + \int_t^T F_i(s, X_s, \hat\a_s, Z^{\xi,\hat\a, i}_s) ds - \int_t^T Z^{\xi, \hat\a, i}_sdB_s,~i=1,\cds, n, \\
\dis Y^{\xi,\hat\a, 0}_t = g_0(X_T - \sum_{i=1}^n \xi_i) + \int_t^T F_0(s, X_s, \hat\a_s, Z^{\xi,\hat\a, 0}_s) ds - \int_t^T Z^{\xi, \hat\a, 0}_sdB_s.
\ea\right.
\eea
Introduce further, for the $\ol F$ in \reff{barYai} and $i=1,\cds, n$,
\bea
\label{PA-barYai}
\left.\ba{c}
\dis \ol Y^{\xi,\hat\a, i}_t = g_i(\xi_i) + \int_t^T \ol F_i(s, X_s,  \hat\a_s, \ol Z^{\xi, \hat\a, i}_s) ds - \int_t^T \ol Z^{\xi,\hat\a, i}_sdB_s.
\ea\right.
\eea

We first characterize the true Nash equilibria of the agents for given $(\xi, \a^0)$, which is defined in an obvious way and denoted as $\cE(\xi,\a^0)$. 
\begin{lem}
\label{lem-PA-Nash}
Let Assumption \ref{assum-0} hold and fix $\xi \in (\dbL^2(\cF_T))^n$ and $\a^0\in \cA$. Then $\a^*$ is in $\cE(\xi,\a^0)$ if and only if 
\bea
\label{PA-NEcharacterization1}
\ol F_i(s, X_s, \a^0_s, \a^*_s, \ol Z^{\xi,  \a^0, \a^*, i}_s) =  F_i(s, X_s, \a^0_s, \a^*_s, \ol Z^{\xi,   \a^0, \a^*, i}_s), ~dt\times d\dbP-a.s., i=1,\cds, n,
\eea
and  if and only if there exist $y\in \dbR^n$ and $\b\in (\dbL^2(\dbF))^n$ such that
\bea
\label{PA-NEcharacterization2}
\left.\ba{c}
\dis \ol F_i(s, X_s, \a^0_s, \a^*_s, \b^{ i}_s) =  F_i(s, X_s, \a^0_s, \a^*_s, \b^{i}_s), ~dt\times d\dbP-a.s.,\ms\\
\dis y_i + \cX^{\a^0, \a^*, \b, i}_T = g_i(\xi_i),~ a.s.,
\ea\right.~ i=1,\cds, n.
\eea
\end{lem}

We next investigate approximate Nash equilibria. Given $R = (R_1, \cds, R_n)$, denote
\bea
\label{PA-cEe}
\cE_\e(\xi, \a^0, R):= \Big\{\a\in \cE_\e(\xi,\a^0): J_i(\xi, \a^0, \a) \ge R_i-\e,~ i=1,\cds, n\Big\}.
\eea
The principal's problem is:
\bea
\label{PA-V0}
\left.\ba{c}
\dis V_0(R) := \lim_{\e\to 0} V_\e(R),\\
\dis V_\e( R) := \sup_{(\xi,\a^0): \cE_\e(\xi, \a^0, R)\neq \emptyset}v_\e(R, \xi,\a^0),~ v_\e(R, \xi,\a^0):= \inf_{ \a^\e\in \cE_\e(\xi, \a^0,  R)} J_0(\xi, \a^0, \a^\e).
\ea\right.
\eea
Denote
\bea
\label{PA-Ierror}
\left.\ba{c}
\dis I_R(\xi, \hat \a, y,\b) := \sum_{i=1}^n \dbE\Big[ \int_0^T |\ol F_i(s, X_s, \hat\a_s, \b^{ i}_s) - F_i(s, X_s, \hat\a_s, \b^{i}_s)|^{3\over 2} ds \\
\dis + \big|y_i + \cX^{\hat\a, \b, i}_T - g_i(\xi_i)\big|^2 + \big((y_i-R_i)^-\big)^2\Big]. 
\ea\right.
\eea

Proposition \ref{prop-Nashe} remains true in this case. 
\begin{prop}
\label{prop-PA-Nashe}
Assume Assumption \ref{assum-0} holds and fix $\xi \in (\dbL^2(\cF_T))^n$, $\a^0\in \cA$,  and $\e>0$. Then there exists $\d>0$ such that the following hold.

\no(i) For any $  \a^\d\in\cE_\d(\xi, \a^0, R)$,  $I_R(\xi, \a^0, \a^\d,  J(\xi, \a^0, \a^\d), \ol Z^{\xi, \a^0, \a^\d}) \le \e$.

\no(ii) For any $  \a\in \cA^n$, if there exist $y$, $\b$ such that $I_R(\xi,  \a^0, \a, y, \b)\le \d$, then $  \a\in \cE_\e(\xi, \a^0, R)$. 
\end{prop}

Next, denote
\bea
\label{PA-Vn}
\left.\ba{c}
\dis V_\l(R):=  \sup_{\xi\in (\dbL^2(\cF_T))^n,\a^0\in \cA} v_\l(\xi, \a^0, R),\\
 \dis \mbox{ where}\q v_\l(\xi,\a^0, R) := \inf_{y\in \dbR^n, \a\in \cA^n,  \b\in (\dbL^2(\dbF))^n}  \Big[J_0(\xi, \a^0, \a) + \l  I_R(\xi, \a^0, \a, y, \b)\Big].
 \ea\right.
\eea

\begin{thm}
\label{thm-PA-Vn}
Assume Assumption \ref{assum-0} holds and $\cE_\e(\xi,\a^0, R) \neq \emptyset$, for any $\xi \in (\dbL^2(\cF_T))^n$, $\a^0\in \cA$,  and $\e>0$.  Then $V_\l(R) \uparrow V_0(R)$ as $\l\uparrow \infty$.
\end{thm}
The proof is almost the same as that for Theorem \ref{thm-Vn} and thus is omitted.

Moreover, note that, 
\bea
\label{PA-miny}
\left.\ba{lll}
\dis \inf_{y_i\in \dbR} \dbE\Big[ \big|y_i + \cX^{\hat\a, \b, i}_T - g_i(\xi_i)\big|^2 + \big((y_i-R_i)^-\big)^2\Big] \\
\dis=\phi\big(R_i, \dbE[g_i(\xi_i)-\cX^{\hat\a, \b, i}_T]\big) + \dbE\big[ \big|\cX^{\hat\a, \b, i}_T - g_i(\xi_i)\big|^2\big],\\
\dis \mbox{where}\q \phi(R_i, x) := \inf_{y_i\in \dbR} \Big[y_i^2 - 2x y_i + \big((y_i-R_i)^-\big)^2\Big].
\ea\right.
\eea
Then, similarly to \reff{Vlambda-2} we have
\bea
\label{PA-Vn-2}
\left.\ba{lll}
\dis V_\l(R) =\sup_{(\xi,\a^0)\in (\dbL^2(\cF_T))^n\times \cA}  \inf_{(\a,\b)\in \cA^n\times (\dbL^2(\dbF))^n}  \Big\{  \l \sum_{i=1}^n \phi\big(R_i, \dbE[g_i(\xi_i)-\cX^{\hat\a, \b, i}_T]\big)  \\
\dis \q + \dbE\Big[M^{\hat\a}_T g_0(X_T-\sum_{i=1}^n \xi_i)+\l \sum_{i=1}^n \big|\cX^{\hat\a, \b, i}_T - g_i(\xi_i)\big|^2  \\
\dis \q + \int_0^T \big[M^{\hat\a}_s f_0(s, X_s,\hat \a_s) ds+ \l \sum_{i=1}^n  |\ol F_i(s, X_s, \hat\a_s, \b^{ i}_s) - F_i(s, X_s, \hat\a_s, \b^{i}_s)|^{3\over 2}\big]ds \Big]\Big\}.
\ea\right.
\eea

\section{Zero sum games with open loop controls:  a special case}
\label{sect-drift}
\setcounter{equation}{0} 

Motivated by \reff{Vlambda-2}, in this section we consider the following zero sum games with open loop controls, which has independent interest and the general model seems new in the literature.
Fix $(\O, \cF, \dbP)$, $B$, $\dbF=\dbF^B$ as before. Let $\cA$ denote the set of $\dbF$-progressively measurable processes. By abusing the notations, consider the following problem:
\bea
\label{drift-V0}
\left.\ba{c}
\dis X^{\hat\a}_t := x + \int_0^t b_s(X^{\hat\a}_s, \hat\a_s) ds + \int_0^t \si_s(X^{\hat\a}_s, \a^0_s) dB_s;\\
 \dis V_0 := \sup_{\a^0\in \cA} V_0(\a^0),\q V_0(\a^0):= \inf_{\a\in \cA} J(\a^0, \a),\q J(\hat\a):= \dbE\Big[g(X^{\hat\a}_T) + \int_0^t f_s(X^{\hat\a}_s, \hat\a_s)ds\Big].
 \ea\right.
 \eea
 Here $\hat\a=(\a^0, \a)$, $b, \si, f$ are $\dbF$-progressively measurable, and $g$ is $\cF_T$-measurable. In particular, $\si$ does not depend on $\a$, that is, the follower controls only the drift.  For notational simplicity, in this section we assume $\a$ is also $1$-dimensional. Introduce the Hamiltonian: 
 \bea
 \label{drift-H}
 H_s(x, z, a_0):= \inf_{a\in A} h_s(x, z, a_0, a),\q  \mbox{where}\q  h_s(x, z,  \hat a):= z b_s(x, \hat a) + f_s(x, \hat a).
 \eea 
 In this section, we enforce the following assumption.

  \begin{assum}
 \label{assum-drift}
 (i)  $b, \si$ are differentiable in $x$;  $f, g$ are twice differentiable in $x$; and all the functions and the above derivatives are bounded, continuous in $t$, and uniformly continuous in $(x, \hat a)$;
 
 \no (ii) $H$ is twice differentiable in $(x, z)$; $H$ and the above derivatives are continuous in $(t,x,z, a_0)$; and the derivatives (but not $H$ itself) are bounded.
  \end{assum}
  We note that, since we will assume further  some strong conditions, we do not optimize the conditions here. Moreover, although we do not assume regularity in $B$ here, we remark that such a requirement is not stringent in our setting because we are using weak formulation in our original problem, and thus $B=X$ is the state process instead of noise, see \reff{X}.

 Similarly as before, we may consider $\e$-optimal controls. In this case, however, this procedure won't change the leader's value. Indeed, for any $\e> 0$ and $\a^0$, let $\cE_\e(\a^0)$ denote the $\e$-optimal control $\a^\e$ { of the optimization problem $V_0(\a^0)$} in \reff{drift-V0}. Denote
 \bea
 \label{drift-Ve}
 \left.\ba{c}
 \dis { v_\e}(\a^0):= \inf_{\a\in \cE_\e(\a^0)} J(\a^0, \a),\q { V_\e} :=  \sup_{\a^0\in \cA} {v_\e}(\a^0),\\
 \dis \ol v_\e(\a^0):= \sup_{\a\in \cE_\e(\a^0)} J(\a^0, \a),\q \ol V_\e :=  \sup_{\a^0\in \cA} \ol v_\e(\a^0)
 \ea\right.
 \eea
 { Note that $J(\a^0, \a) -  V_0(\a^0)>\e$ for $\a\notin \cE_\e(\a^0)$, then $\inf_{\a\in \cE(\a^0)} J(\a^0, \a) =  \inf_{\a\in \cE_\e(\a^0)} J(\a^0, \a)$. Moreover, $0\le J(\a^0, \a) -  V_0(\a^0) \le \e$ for any $\a\in \cE_\e(\a^0)$.  Then it} is clear that 
 \bea
 \label{drift-Ve2}
 \left.\ba{c}
\dis{ v_\e(\a^0)= V_0(\a^0),\q V_\e = V_0,}\ms\\ 
\dis { 0\le \ol v_\e(\a^0) - V_0(\a^0)\le \e,\q 0\le \ol V_\e - V_0 \le \e,\q \mbox{and thus}\q \lim_{\e\to 0} \ol V_\e = V_0.}
\ea\right.
 \eea
 
The problem \reff{drift-V0} is a special case of our original Stackelberg game \reff{leader} with one follower, except that we are using open loop controls, instead of closed loops controls as in Section \ref{sect-stackelberg}.  We remark that in this case, even if $\si$ is nondegenerate,  it is not convenient to use weak formulation anymore because $\hat\a$ depends on $B$, not on $X^{\hat\a}$.  
 
 We first characterize  $V_0(\a^0)$ for given $\a^0$, which is a standard control problem. Consider the following BSPDE (or stochastic HJB equation):
  \bea
 \label{drift-BSPDE}
 \left.\ba{lll}
 \dis u^{\a^0}_t(x) = g(x) - \int_t^T v_s^{\a^0}(x)dB_s \\
 \dis\q + \int_t^T \big[{1\over 2} \pa_{xx} u^{\a^0}_s(x) \si^2_s(x, \a^0_s)+ \pa_{x} v^{\a^0}_s(x) \si_s(x, \a^0_s)+ H_s(x, \pa_x u_s^{\a^0}(x), \a^0_s)\big] ds.
 \ea\right.
 \eea
 
 \begin{defn}
(i) For $k\ge 1$, let $C^k(\dbF)$ denote the set of $\dbF$-progressively measurable random fields $u: [0, T]\times \dbR\times \O\to \dbR$ such that $u$ is $k$-th differentiable in $x$ and all the derivatives (including $u$ itself) are bounded and continuous in $(t,x)$. 

\no(ii) Let $\cA_0$ denote the set of $\a^0\in \cA$ such that  BSPDE \reff{drift-BSPDE} has a classical solution  $(u^{\a^0}, v^{\a^0})\in C^3(\dbF)\times C^2(\dbF)$.
\end{defn}

The following result is standard, see e.g. \cite{Peng}.
 \begin{prop}
 \label{prop-drift-supa}
 Let Assumption \ref{assum-drift} hold. For any $\a^0\in \cA_0$, we have $V_0(\a^0) = u_0^{\a^0}(x)$.
 \end{prop}
 
 \begin{rem}
 \label{rem-drift-BSPDE}
 (i) Assume $\f_t(x) = \f(t, x, B_t)$ for $\f=b, \si, f, \a^0$, and  $g(x) = g(x, B_T)$, where, by abusing the notations, the functions in the right side are deterministic. Then we have
 \beaa
 u^{\a^0}_t(x) = w(t, x, B_t), \q v^{\a^0}_t(x) = \pa_y w(t, x, B_t),
 \eeaa
 and $w(t,x,y)$ satisfies the following degenerate PDE: denoting $\hat\si(t,x,y) := \si\big(t,x,y, \a^0(t,x, y)\big)$,
 \beaa
&& \pa_t w + {1\over 2} \pa_{yy} w +{1\over 2} \pa_{xx} w \hat\si^2+ \pa_{xy} w \hat \si+ H\big(t, x, y, \pa_x w, \a^0(t,x,y)\big) =0,\\
 && w(T,x,y) = g(x, y).
 \eeaa
 
\no (ii) In the general case: $\f_t(x) = \f_t(x, B_{\cd\wedge t})$ for $\f=b, \si, f, \a^0$, and  $g(x) = g(x, B)$, we have 
 \beaa
 u^{\a^0}_t(x) = W(t, x, B_{\cd\wedge t}), \q v^{\a^0}_t(x) = \pa_\by W(t, x, B_{\cd\wedge t}),
 \eeaa
 where $\pa_\by W(t,x,\by)$, $\by\in C([0, T])$, denotes the Dupire's derivative of $W$ with respect to the path $\by$, and $W$ satisfies the following path dependent PDE: denoting $\hat\si(t,x,\by) := \si\big(t,x,\by, \a^0(t,x, \by)\big)$,
 \beaa
&& \pa_t W + {1\over 2} \pa_{\by\by} W +{1\over 2} \pa_{xx} W \hat\si^2+ \pa_{x\by} W \hat \si+ H\big(t, x, \by, \pa_x W, \a^0(t,x,\by)\big) =0,\\
 && W(T,x,\by) = g(x, \by).
 \eeaa
 We refer to \cite{Dupire, CF} for the Duprie's derivatives and \cite[Section 11.3.5]{Zhang-book} for the above connection. We note that, here $B$ denotes the state process $X$, see \reff{X}, rather than the noise. So it is reasonable to assume certain regularity in $B$.
 \end{rem}
 
 We next provide another characterization for $V_0(\a^0)$, by using coupled FBSDEs.
 { First, for fixed $\a^0$, introduce the characteristic SDE}:\footnote{ When  the Hamiltonian $H$ in \reff{drift-H} has an optimal control $a^*$, by the envelope theorem we have
\bea
\label{drift-a*}
\left.\ba{c}
\dis b_s(x, \a^0_s, a^*)= \pa_z H_s(x, z, a_0),\q f_s(x, a_0, a^*) =   \big[H_s - z \pa_z H_s\big](x, z,   a_0).
 \ea\right.
\eea
}
\bea
\label{drfit-SDE1}
X^{\a^0}_t := x + \int_0^t \pa_zH_s(X^{\a^0}_s, \pa_x u_s^{\a^0},  \a^0_s) ds + \int_0^t  \si_s(X^{\a^0}_s,   \a^0_s) dB_s.
 \eea
 Denote 
 \beaa
 Y^{0,\a^0}_t := u_t^{\a^0}(X^{\a^0}_t),\q Y^{1,\a^0}_t := \pa_x u_t^{\a^0}(X^{\a^0}_t).
 \eeaa 
 By differentiating \reff{drift-BSPDE} we have, denoting $(u, v) = (u^{\a^0}, v^{\a^0})$,
  \bea
 \label{drift-BSPDE2}
 \left.\ba{lll}
 \dis \pa_x u_t(x) = \pa_x g  - \int_t^T \pa_x vdB_s + \int_t^T \big[{1\over 2} \pa_{xxx} u \si^2+ \pa_{xx} v\si + \pa_z H \pa_{xx} u\big] ds\\
 \dis \qq\qq + \int_t^T \big[\pa_{xx} u \si\pa_x \si +\pa_x v \pa_x \si  + \pa_x H \big] ds.
 \ea\right.
 \eea
 Applying the It\^{o}-Ventzell formula we have
\bea
\label{drift-BSDEs}
\left.\ba{lll}
\dis  d Y^{0,\a^0}_t = Z^{0,\a^0}_t dB_t - \big[H_t - \pa_x u_t^{\a^0} \pa_z H_t\big](X^{\a^0}_t, \pa_x u_t^{\a^0}, \a^0_t)dt; \\
\dis  d Y^{1,\a^0}_t = Z^{1,\a^0}_t dB_t - \big[\pa_{xx} u \si\pa_x \si +\pa_x v \pa_x \si  + \pa_x H \big]dt.
 \ea\right.
\eea
where, 
\bea
\label{drift-Z}
\dis Z^{0,\a^0}_t = v^{\a^0}_t(X^{\a^0}_t) + \pa_x u_t\si_t(X^{\a^0}_t, \a^0_t), \q Z^{1,\a^0}_t = \pa_x v^{\a^0}_t(X^{\a^0}_t) + \pa_{xx} u_t\si_t(X^{\a^0}_t, \a^0_t).
\eea
Then we obtain the following FBSDE
\bea
\label{drift-FBSDE}
\left.\ba{lll}
\dis X^{\a^0}_t := x + \int_0^t \pa_zH_s(X^{\a^0}_s, Y^{1,\a^0}_s, \a^0_s) ds + \int_0^t \si_s(X^{\a^0}_s, \a^0_s) dB_s;\\
\dis Y^{0,\a^0}_t = g(X^{\a^0}_T) +\int_t^T \big[H_s - Y^{1,\a^0}_s \pa_z H_s\big](X^{\a^0}_s, Y^{1,\a^0}_s, \a^0_s)ds - \int_t^T Z^{0,\a^0}_s dB_s; \\
\dis Y^{1,\a^0}_t = \pa_x g(X^{\a^0}_T) +\int_t^T \big[Z^{1,\a^0}_s\pa_x \si_s(X^{\a^0}_s, \a^0_s)   + \pa_x H_s(X^{\a^0}_s, Y^{1,\a^0}_s, \a^0_s) \big]ds -\int_t^T Z^{1,\a^0}_s dB_s.
 \ea\right.
 \eea
 We remark that at above the first and third equations are coupled, while the second one is decoupled from the other two.

\begin{assum}
 \label{assum-drift2}
 The set $\cA_0$ is dense in $\cA$. That is, for any $\a^0\in \cA$, there exists $\tilde \a^{0,m}\in \cA_0$ such that $\tilde \a^{0,m}\to \a^0$ in measure $dt\times d\dbP$, as $m\to\infty$.
 \end{assum}

\begin{prop}
 \label{prop-drift-supa2}
 Let Assumption \ref{assum-drift} hold. 
 
\no (i) For any $\a^0\in \cA_0$,  the above FBSDE is well-posed, and $V_0(\a^0) = Y^{0,\a^0}_0$. 

\no (ii) Moreover, if Assumption \ref{assum-drift2} holds, then  $V_0 = \sup_{\a^0\in \cA_0} Y^{0,\a^0}_0$
 \end{prop}

\proof (i) follows from the same arguments in the four step scheme, see \cite{MPY}, by replacing the PDE there with BSPDE \reff{drift-BSPDE}. We omit the details.  

To see (ii), denote $\tilde V_0 := \sup_{\a^0\in \cA_0} Y^{0,\a^0}_0$. First it is clear that $\tilde V_0 \le V_0$. On the other hand, fix any $\a^0\in \cA$, let $\tilde\a^{0,m}\in \cA_0$ be an approximating sequence as in Assumption \ref{assum-drift2}. Recalling \reff{drift-V0} and by the uniformly continuity in Assumption \ref{assum-drift} (i), it is clear that
\beaa
\lim_{m\to\infty}\sup_{\a\in \cA} |J(\tilde \a^{0,m}, \a) - J(\a^0, \a)| =0.
\eeaa
Then $V(\a^0) \le \liminf_{m\to\infty} V(\tilde \a^{0,m}) \le \tilde V_0$. This implies $V_0\le \tilde V_0$ immediately. 
\qed

We now optimize $Y^{0, \a^0}_0$ over $\a^0\in \cA_0$. We first rewrite  the last equation of \reff{drift-FBSDE} in a forward form. That is, given $\a^0$ and $y, \b$, introduce
\bea
\label{drift-forward}
\left.\ba{lll}
\dis X^{\a^0, y, \b}_t = x + \int_0^t \pa_zH_s(X^{\a^0, y, \b}_s, Y^{\a^0, y, \b}_s, \a^0_s) ds + \int_0^t \si_s(X^{\a^0, y, \b}_s, \a^0_s) dB_s;\\
\dis Y^{\a^0, y, \b}_t =  y - \int_0^t \big[\b_s\pa_x \si_s(X^{\a^0, y, \b}_s, \a^0_s)   + \pa_x H_s(X^{\a^0, y, \b}_s, Y^{\a^0, y, \b}_t , \a^0_s) \big]ds +\int_0^t \b_s dB_s;\\
\dis J(\a^0, y, \b) = \dbE\Big[g(X^{\a^0, y, \b}_T) +\int_0^T \big[H_s - Y^{\a^0, y, \b}_s \pa_z H_s\big](X^{\a^0, y, \b}_s, Y^{\a^0, y, \b}_s, \a^0_s)ds\Big].
 \ea\right.
 \eea
 Then, by Proposition \ref{prop-drift-supa2}, we have\footnote{Given $\a^0$, the $(y, \b)$ satisfying the constraint is actually unique. However, for the convenience of \reff{drift-Vkappalambda} below, we use $\sup_{y, \b}$ here.}
 \bea
 \label{drift-Vkappa2}
\dis V_0 = \sup_{\a^0\in \cA_0}\sup_{(y,\b)\in \dbR\times \dbL^2(\dbF)} J(\a^0, y, \b),\q\mbox{subject to}\q Y^{\a^0, y, \b}_T = \pa_x g(X^{\a^0,y,\b}_T).
 \eea
We next introduce the penalization to get rid of the constraint, for $\l>0$ large:
  \bea
 \label{drift-Vkappalambda}
 \left.\ba{c}
\dis  V_{\l} = \sup_{\a^0\in \cA}\sup_{(y,\b)\in \dbR\times \dbL^2(\dbF)}  J_\l(\a^0, y, \b),\\
\dis\mbox{where}\q J_\l(\a^0, y, \b) := J(\a^0, y, \b) -\l \dbE\Big[\big|Y^{\a^0, y, \b}_T - \pa_x g(X^{\a^0,y,\b}_T)\big|^2\Big].
 \ea\right.
 \eea
 Note that we consider $\a^0\in \cA$ instead of $\cA_0$ at above.
  
 \begin{assum}
 \label{assum-drift3}
 For any $\l \ge 1$, \reff{drift-Vkappalambda} has ${1\over \l}$-optimal control $(\a^{0,\l}, y^\l, \b^\l)$\footnote{ We note that \reff{drift-Vkappalambda} is a control problem, and thus ${1\over \l}$-optimal control always exists in $\cA\times \dbR\times \dbL^2(\dbF)$. The assumption here is about the stronger requirements on $(\a^{0,\l}, y^\l, \b^\l)$.} such that $\a^{0,\l}\in \cA_0$ and $(u^{\a^{0,\l}}, v^{\a^{0,\l}})\in C^3(\dbF)\times C^2(\dbF)$ with the corresponding derivatives uniformly bounded, uniformly in $\l$. 
 \end{assum}
  \begin{thm}
 \label{thm-drift-VKappalambda}
(i) Let Assumptions \ref{assum-drift} and \ref{assum-drift2} hold. Then $V_0\le V_\l$ for all $\l>0$.
 
\no(ii) Assume further that Assumption \ref{assum-drift3} holds. Then $\lim_{\l\to\infty} V_{\l} = V_0$.
 \end{thm}
 \proof For  any $(\a^0, y, \b)\in \cA_0\times \dbR^n\times (\dbL^2(\dbF))^n$ satisfying the constraint in \reff{drift-Vkappa2}, clearly $J_\l(\a^0, y, \b) = J(\a^0, y, \b)$. Then $V_0\le V_{\l}$ for any $\l$. On the other hand, fix $\l\ge 1$ and choose ${1\over \l}$-optimal control $(\a^{0}, y, \b)=(\a^{0,\l}, y^\l, \b^\l)$ satisfying the requirements in Assumption \ref{assum-drift3}. Then 
 \bea
 \label{drift-lambda}
 J_\l(\a^0, y, \b) \ge V_{\l}-{1\over \l}\ge V_0-1.
 \eea
   Since $f$ and $g$ are bounded, then
 \beaa
 V_0 - 1\le J(\a^0, y, \b) -\l \dbE\Big[\big|Y^{\a^0, y, \b}_T - \pa_x g(X^{\a^0,y,\b}_T)\big|^2\Big]\le C -\l \dbE\Big[\big|Y^{\a^0, y, \b}_T - \pa_x g(X^{\a^0,y,\b}_T)\big|^2\Big].
 \eeaa
 Thus
 \beaa
 \dbE\Big[\big|Y^{\a^0, y, \b}_T - \pa_x g(X^{\a^0,y,\b}_T)\big|^2\Big] \le {C\over \l}.
 \eeaa
 Notice that
 \beaa
\left.\ba{lll}
\dis X^{\a^0, y, \b}_t = x + \int_0^t \pa_zH_s(X^{\a^0, y, \b}_s, Y^{\a^0, y, \b}_s, \a^0_s) ds + \int_0^t \si_s(X^{\a^0, y, \b}_s, \a^0_s) dB_s;\\
\dis Y^{\a^0, y, \b}_t =  \pa_x g(X^{\a^0,y,\b}_T) + [Y^{\a^0, y, \b}_T - \pa_x g(X^{\a^0,y,\b}_T)]  -\int_t^T \b_s dB_s\\
\dis\qq\qq + \int_t^T \big[\b_s\pa_x \si_s(X^{\a^0, y, \b}_s, \a^0_s)   + \pa_x H_s(X^{\a^0, y, \b}_s, Y^{\a^0, y, \b}_t , \a^0_s) \big]ds,
 \ea\right.
 \eeaa
 and that the first and third equations in \reff{drift-FBSDE} form a self-contained coupled FBSDE. The well-posedness of the BSPDE \reff{drift-BSPDE} implies the well-posedness of the latter FBSDE, and that
 \beaa
&&\dis \dbE\Big[\sup_{0\le t\le T} [|X^{\a^0, y, \b}_t -X^{\a^0}_t|^2 +  |Y^{\a^0, y, \b}_t -Y^{1,\a^0}_t|^2\big] + \int_0^T |\b_t -Z^{1,\a^0}_t|^2dt\Big]\\
 &&\dis \le C\dbE\Big[\big|Y^{\a^0, y, \b}_T - \pa_x g(X^{\a^0,y,\b}_T)\big|^2\Big] \le {C\over \l}.
 \eeaa
 In particular, the constant $C$ is independent of $\a^0=\a^{0,\l}$, and hence independent of $\l$, due to Assumption \ref{assum-drift3}. Then, by the second equation of \reff{drift-FBSDE} and the last equation of \reff{drift-forward},
 \beaa
 |Y^{0,\a^0}_0-J(\a^0, y, \b)|\le {C\over \sqrt{\l}}.
 \eeaa
 By \reff{drift-lambda} and \reff{drift-Vkappalambda}, this implies that
\beaa
V_{\l} \le J_\l(\a^0, y, \b) + {1\over \l} \le J(\a^0, y, \b) + {1\over \l} \le Y^{0, \a^0}_0 + {C\over \sqrt{\l}} + {1\over \l}\le V_0 + {C\over \sqrt{\l}}.
\eeaa
Thus $\lim_{\l\to\infty} V_{\l} \le V_0$, and hence equality holds.
\qed
 
 Finally we compute $V_{\l}$. Clearly 
 \bea
 \label{drift-Vly}
 V_\l=\sup_{y\in \dbR} V_\l(y),\q V_\l(y) := \sup_{\a^0\in \cA, \b\in \dbL^2(\dbF)}  J_\l(\a^0, y, \b).
 \eea
Then the following result is obvious.
\begin{thm}
\label{thm-drift-leader1}
Let Assumptions \ref{assum-drift},  \ref{assum-drift2}, and \ref{assum-drift3} hold, and the following BSPDE has a classical solution $\big(u^\l_t(x, y), v^\l_t(x, y)\big)$: with $\si = \si_t(x, a_0)$ and $H = H_t(x,y, a_0)$,
\bea
\label{drift-HJB1}
\left.\ba{lll}
\dis d u^\l_t(x, y) =  v^\l_t(x, y) dB_t - \sup_{a_0\in A, \b\in \dbR} \Big[{1\over 2}\pa_{xx} u^\l \si^2_t  +{1\over 2} \pa_{yy} {  u^\l} \b^2 + \pa_{xy} {  u^\l} \si_t \b \\
\dis \qq + \pa_x u^\l \pa_z H  - \pa_y{  u^\l} [\b \pa_x\si + \pa_x H] + \pa_x v^\l \si + \pa_y v^\l \b + H - y\pa_z H\Big]dt, \\
\dis u^\l_T(x, y) = g(x) - \l |y  - { \pa_x g(x)}|^2.
\ea\right.
\eea
Then $V_{\l}(y) = u^{\l}_0(x, y)$.
\end{thm}

We remark that in this paper we focus on the structure, thus at above we assume Assumption \ref{assum-drift3} and the uniform well-posedness of BSPDE \reff{drift-BSPDE}, which are very strong requirements. We shall try to weaken them in future research.

\begin{rem}
\label{rem-drift-leader1}
(i) The value $V_{\l}$ in \reff{drift-Vkappalambda} is well-defined without assuming the the well-posedness of BSPDE \reff{drift-BSPDE}. It will be interesting to explore weaker conditions under which $\lim_{\l\to\infty} V_{\l} = V_0$.

\no (ii) Since $\b$ is unbounded, in general the BSPDE \reff{drift-HJB1} may have facelifting issue, namely it is possible that {$ u_{\l}(T-, x, y) \neq g(x) - \l |y  -  \pa_x g(x)|^2$. In this case, we shall first compute $u_{\l}(T-, x, y) $ and then solve \reff{drift-HJB1} with terminal condition $u_{\l}(T-, x, y) $,} see e.g. \cite{BCS, ST}. We do not investigate this issue in this paper.

{
\no (iii) Although the control $\b$ is unbounded, the Hamiltonian in \reff{drift-HJB1} is still finite when $t<T$. Indeed, the term under $\sup$ is quadratic in $\b$, and the unboundedness of $\b$  implies that $u^\l$ is concave in $y$. The existence of classical solution, however, is of course a strong requirement.   
}
\end{rem}

\begin{rem}
\label{rem-drift-leader2}
Since the optimization problem $V_\l(y)$ in \reff{drift-Vly} has open loop controls $\a^0, \b$, one can also apply the stochastic maximum principle to derive the first order conditions. We leave this to interested readers.
\end{rem}

 \section{Zero sum games with open loop controls: the general case}
\label{sect-vol}
\setcounter{equation}{0} 
In this section, we consider the following problem with mean field controls:
\bea
\label{vol-V0}
\left.\ba{c}
\dis X^{\hat\a}_t := x + \int_0^t b(\L^{\hat\a}_s, \hat\a_s) ds + \int_0^t \si(\L^{\hat\a}_s, \hat\a_s)dB_s,\q\mbox{where}~ \L^{\hat\a}_t := \big(t, \cL_{(B_{\cd\wedge t}, X^{\hat \a}_t)},B_{\cd\wedge t}, X^{\hat \a}_t\big), \\
 \dis V_0 := \sup_{\a^0\in \cA}\inf_{\a\in \cA} J(\hat\a), \q J(\hat\a):=\dbE\Big[ g(\L_T^{\hat\a}) + \int_0^T f(\L_s^{\hat\a}, \hat\a_s)ds\Big].
 \ea\right.
 \eea 
 In light of \reff{Vlambda-2}, here $X = (X^1, \cds, X^d)$ are $d$-dimensional, but for notational simplicity, we still assume $ B$ is one dimensional, and thus $b, \si$ are $\dbR^d$-valued. We denote $\bX_{\cd\wedge s} := (B_{\cd\wedge s}, X_s)$, and $b, \si, f, g$ are deterministic functions. Denote $\dbX^d := C([0, T])\times \dbR^d$, and $\bx = (\bx^0, x)\in \dbX^d$, $\bx_{\cd\wedge t} := (\bx^0_{\cd\wedge t}, x)$. We say a function $\f$ on $[0, T]\times \cP_2(\dbX^d)\times \dbX^d $ is adapted if 
 \beaa
 \f(t, \cL_\bX, \bx) = \f(t, \cL_{\bX_{\cd\wedge t}}, \bx_{\cd\wedge t}).
 \eeaa
 Throughout this section, all the path dependent functions are required to be adapted, and thus we don't have to write $_{\cd\wedge t}$ explicitly. 
 
  \begin{assum}
 \label{assum-vol1}
 The coefficients $b, \si, f, g$ are bounded, adapted, and continuous in all variables, and  uniformly Lipschitz continuous in $(\mu, \bx)\in \cP_2(\dbX^d)\times \dbX^d$.
 \end{assum}
 We remark that, although $B$ is a Brownian motion, it stands for the state process in the weak formulation of the original leader follower problem, see \reff{X}. So it is reasonable to assume the  Lipschitz continuity of the coefficients with respect to $(\mu, \bx)$. However, as mentioned in the previous section,  in this paper we focus on the structure, and thus we may assume very strong conditions later in this section. We shall try to weaken them in future research,
 
Under Assumption \ref{assum-vol1}, we may derive a counterpart of Proposition \ref{prop-drift-supa}. However, due to the volatility control, in this case we are not able to derive the FBSDE characterization as in Proposition \ref{prop-drift-supa2}. For this purpose, we introduce an additional randomness. Let $B^1, \cds, B^d$ be independent Brownian motions, and be independent of $B^0:= B$. Set $\dbF' := \dbF^{B^0, B^1, \cds, B^d}$, and let $\cA'$ denote the set of $\dbF'$-progressively measurable processes. Fix a small parameter $\k>0$, for $i=1,\cds, d$, introduce, { with $\bX^{\k,\hat\a} = (B^0, X^{\k,\hat\a,1}, \cds, X^{\k,\hat\a, d})$, }
 \bea
\label{vol-Vk}
\left.\ba{lll}
\dis X^{\k,\hat\a,i}_t := x_i + \int_0^t b_i(\L_s^{\k,\hat\a}, \hat\a_s) ds + \int_0^t \si_i(\L_s^{\k,\hat\a}, \hat\a_s)dB_s + \k B^i_s,\q \L_s^{\k,\hat\a} :=\big(s,  \cL_{\bX^{\k,\hat\a}_{\cd\wedge s}}, \bX^{\k,\hat\a}_{\cd\wedge s}\big), \\
 \dis V_\k := \sup_{\a^0\in \cA} V_\k(\a^0),\q V_\k(\a^0):=\inf_{\a\in \cA'} J_\k(\hat\a),\q J_\k(\hat\a):=\dbE\Big[ g(\L_T^{\k,\hat\a}) + \int_0^T f(\L_s^{\k,\hat\a}, \hat\a_s)ds\Big].
 \ea\right.
 \eea 
 We emphasize that $\a$ is $\dbF'$-measurable, but $\a^0$ is still $\dbF$-measurable.  
 
 \begin{lem}
 \label{lem-vol-Vk}
 Under Assumption \ref{assum-vol1}, we have $|V_\k - V_0|\le C\k$.
 \end{lem}
 \proof First, since $B^1, \cds, B^d$ are independent of $B$, it is clear that 
 \beaa
 \sup_{\a^0\in \cA}\inf_{\a\in \cA'} J(\a^0, \a) = V_0.
 \eeaa
 Moreover, for any $\a^0\in \cA$, $\a\in \cA'$, by standard estimates we have
 \beaa
 \dbE\big[\sup_{0\le t\le T}|X^{\k,\hat\a}_t - X^{\hat\a}_t|^2\big]\le C\k^2,\q |J_\k(\hat\a) - J(\hat\a)|\le C\k.
 \eeaa
 This clearly implies $|V_\k-V_0|\le \sup_{\a^0\in \cA, \a\in \cA'}|J_\k(\hat\a) - J(\hat\a)| \le C\k$.
 \qed

  Due to the involvement of $\cL_{\bX}$, clearly the BSPDE \reff{drift-BSPDE} should be replaced with an HJB equation on Wasserstein space of probability measures. Let ${\d\over \d\mu}$ denote the linear functional derivative, and $\pa_{\bx^0}$ the Dupire derivative, see \cite{WZ}. The following result is standard for mean field control problems. 
 \begin{prop}
 \label{prop-vol-supa}
 Let Assumption \ref{assum-vol1} hold and fix $\a^0\in \cA$.  Assume  the following HJB equation on $[0, T]\times \cP_2(\dbX^d)$ has a classical solution  $U^{\k,\a^0}(t, \mu)$ with bounded derivatives:
 \bea
 \label{vol-BPDE}
 \left.\ba{lll}
 \dis \pa_t U(t, \mu) + \int_{\dbX^d} \Big[ {1\over 2} \pa_{\bx^0\bx^0} {\d U\over \d\mu}+ H\big(\cd, \pa_x{\d U\over \d\mu}, \pa_{\bx^0 x}{\d U\over \d\mu}, \pa_{xx}{\d U\over \d\mu}, \a^0(s, \bx^0)\big)\Big](t, \mu, \bx)  \mu(d\bx)=0,\\
 \dis U(T, \mu) = \int_{\dbX^d}g(\mu, \bx)\mu(d\bx),
 \ea\right.
 \eea
 where, for $z\in \dbR^{d}$, $\g^0\in \dbR^d$, and symmetric matrix $\g\in \dbR^{d\times d}$, denoting $\hat\g := (\g^0, \g)$,
 \bea
 \label{vol-H}
 \left.\ba{lll}
\dis H(t, \mu, \bx, z, \hat\g, a_0) := \inf_{a\in A} h(t, \mu,\bx,z,  \hat\g, a_0, a),\\
 \dis :=\inf_{a\in A} \Big[ \sum_{1\le i<j\le d} \g_{ij} \si_i(\cd)\si_j(\cd)  + \sum_{i=1}^d\big[{1\over 2} \g_{ii} [\si_i^2(\cd) +\k^2] + \g^0_{i}\si_i(\cd) +  z_ib_i(\cd)\big] + f(\cd)\Big](t,\bx, \mu, \hat a).
 \ea\right.
 \eea
Then $V_\k(\a^0) = U^{\k,\a^0}(0, \d_{(0,x)})$.
 \end{prop}
 
 Note that, given $(t,\mu, \bx, z, \hat \g, a_0)$, the optimal $a^*$ { (assuming its existence)} of the above Hamiltonian $H$  satisfies:
\bea
\label{vol-a*}
\left.\ba{c}
\dis b(s, \mu, \bx,  a_0, a^*)= \pa_{z} H(s,\mu, \bx, z, \hat\g, a_0),\q \si(s, \bx, \mu, a_0, a^*)= \pa_{\g^0} H(s,\mu, \bx, z, \hat\g, a_0),\\
 \dis f(s,\mu, \bx,  a_0, a^*) =   F(s,\mu, \bx, z, \hat\g, a_0) := \Big[H - \sum_{1\le i<j\le d} \g_{ij}  \pa_{\g^0_{i}} H\pa_{\g^0_{j}} H\\
 \dis\qq \qq - \sum_{i=1}^d \big[{1\over 2} \g_{ii} [{ (\pa_{\g^0_{i}} H)^2} +\k^2] + \g^0_{i}\pa_{\g^0_{i}} H +  z_i\pa_{z_i} H\big]\Big](s,\mu, \bx, z, \hat\g, a_0).
 \ea\right.
\eea
Introduce an SDE system, for $i=1,\cds, { d}$ and for $U:= U^{\k, \a^0}$:
\bea
\label{vol-SDE}
\left.\ba{c}
\dis X^{\k, \a^0, i}_t = x_i + \int_0^t \pa_{z_i}H\Big(\L^{\k,\a^0}_s, \pa_x {\d U\over \d \mu}(\L^{\k,\a^0}_s), \pa_{\bx^0 x} {\d U\over \d \mu}(\L^{\k,\a^0}_s), \pa_{xx} {\d U\over \d \mu}(\L^{\k,\a^0}_s), \a^0_s\Big)ds\\
\dis + \int_0^t  \pa_{\g^0_{i}}H\Big(\L^{\k,\a^0}_s, \pa_x {\d U\over \d \mu}(\L^{\k,\a^0}_s), \pa_{\bx^0 x} {\d U\over \d \mu}(\L^{\k,\a^0}_s), \pa_{xx} {\d U\over \d \mu}(\L^{\k,\a^0}_s), \a^0_s\Big) dB_s + \k B^i_t,
\ea\right.
 \eea
where $\L^{\k,\a^0}_s:=\big(s, \cL_{\bX^{\k,\a^0}_{\cd\wedge s}}, \bX^{\k, \a^0}_{\cd\wedge s}\big)$, and denote
\bea
\label{vol-Y}
Y^{\k, \a^0}_t := \pa_x {\d U^{\k, \a^0}\over \d\mu} (\L^{\k,\a^0}_t),\q G(\mu, \bx) = \pa_{x}g(\mu, \bx) +  \int_{\dbX^d}\pa_x{\d g\over \d\mu}(\mu, \bx, \tilde\bx)\mu(d\tilde \bx).
\eea

To derive $d Y^{\k, \a^0}_t$, we differentiate \reff{vol-BPDE} with respect to $\pa_x {\d\over \d\mu}$ and denote $v_i := \pa_{x_i}{\d U\over \d\mu}$:
  \bea
 \label{vol-BPDEi}
 \left.\ba{lll}
 \dis \pa_t v_i(t, \mu, \bx) + {1\over 2} \pa_{\bx^0\bx^0} v_i(t,\mu, \bx)  + \Big[\pa_{x_i} H(\cd) + \pa_z H(\cd)\cd \pa_x v_i + \pa_{\g^0}H(\cd) \cd \pa_{\bx^0 x} v_i \\
 \dis\qq\qq + \pa_\g H(\cd) : \pa_{xx} v_i\Big]\big(t, \mu, \bx, \pa_x {\d U\over \d\mu}, \pa_{\bx^0 x}{\d U\over \d\mu}, \pa_{xx} {\d U\over \d\mu},  \a^0(s, \bx^0)\big) \\
 \dis +  \int_{\dbX^d} \pa_{x_i}{\d H\over \d\mu}  \big(t, \mu, \bx, \tilde\bx, \pa_x {\d U\over \d\mu}(t, \mu,\tilde\bx),\pa_{\bx^0 x}{\d U\over \d\mu}(t, \mu, \tilde\bx), \pa_{x x}{\d U\over \d\mu}(t, \mu, \tilde\bx), \a^0(s, \tilde \bx^0)\big) \mu(d\tilde\bx)\ms\\
 \dis +\int_{\dbX^d} \!\! \Big[ \pa_z H(\cd) \cd \pa_x {\d v_i\over \d\mu}(t, \mu, \bx, \tilde \bx) + \pa_{\g^0} H(\cd) \cd \pa_{\bx^0x} {\d v_i\over \d\mu}(t, \mu, \bx, \tilde \bx) +  \pa_{\g} H(\cd) : \pa_{xx} {\d v_i\over \d\mu}(t, \mu, \bx, \tilde \bx)  \Big]\\
 \dis \qq\qq \big(t, \mu, \tilde\bx, \pa_x {\d U\over \d\mu}(t, \mu,\tilde\bx),\pa_{\bx^0 x}{\d U\over \d\mu}(t, \mu, \tilde\bx), \pa_{x x}{\d U\over \d\mu}(t, \mu, \tilde\bx), \a^0(s, \tilde \bx^0)\big) \mu(d\tilde\bx)=0,\\
 \dis v_i(T, \mu,\bx) =  G_i(\mu, \bx),
 \ea\right.
 \eea
 Applying the It\^{o} formula and by straightforward calculation, we have
\bea
\label{vol-BSDE}
\left.\ba{lll}
\dis  d  Y^{\k, \a^0, i}_t = Z^{\k,\a^0, 0, i}_t dB_t + \sum_{j=1}^d \k Z^{\k, \a^0, i, j}_t dB^j_t   \\
\dis\q - \pa_{x_i} H\Big(\L^{\k,\a^0}_t, \pa_x {\d U\over \d \mu}(\L^{\k,\a^0}_t), \pa_{\bx^0 x} {\d U\over \d \mu}(\L^{\k,\a^0}_t), \pa_{xx} {\d U\over \d \mu}(\L^{\k,\a^0}_t), \a^0_s\Big)dt,
\ea\right.
\eea
where,   
\bea
\label{vol-Z}
 \left.\ba{c}
 \dis Z^{\k,\a^0, 0, i}_t  =\pa_{\bx^0} v_i(\L^{\k,\a^0}_t)+  \pa_{x} v_i(\L^{\k,\a^0}_t) \cd \pa_{\g^0} H,\q Z^{\k,\a^0, i,j}_t =  \pa_{x_j} v_i(\L^{\k,\a^0}_t),\q 1\le i, j\le d,\ms\\
 \dis Z^{\k, \a^0, 0} := \{Z^{\k, \a^0, 0,i}\}_{1\le i\le d},\q Z^{\k, \a^0} := \{Z^{\k, \a^0, i, j}\}_{1\le i, j\le d},\q \hat Z^{\k, \a^0} := \big(Z^{\k, \a^0, 0}, Z^{\k, \a^0}\big).
 \ea\right.
\eea
This implies that
\bea
\label{vol-uZ}
v(\L^{\k,\a^0}_t) = Y^{\k,\a^0}_t,\q  \pa_{x_j} v_i(\L^{\k,\a^0}_t) = Z^{\k,\a^0, i, j}_t. 
 \eea
 Moreover,  by \reff{vol-Z} we have
 \bea
 \label{vol-q}
 \pa_{\bx^0} v_i(\L^{\k,\a^0}_t) + \sum_{j=1}^d Z^{\k,\a^0, i, j}_t \pa_{\g^0_{j}} H\Big(\L^{\k,\a^0}_t,  Y^{\k,\a^0}_t, \pa_{\bx^0} v(\L^{\k,\a^0}_t),   Z^{\k,\a^0}_t, \a^0_t\Big) = Z^{\k,\a^0,0, i}_t.
 \eea
  To obtain a representation of $\pa_{\bx^0} v_i$ from above, we impose the following assumption. 
 
 \begin{assum}
 \label{assum-vol2}
 The mapping $q\in \dbR^d\mapsto q'\in \dbR^d$ with 
 $$
 q'_i:= q_i+  \sum_{j=1}^d \g_{i, j} \pa_{\g^0_{j}} H\big(t,\mu, \bx,  z, q,  \g, a_0\big),\q z\in \dbR^d, \g\in \dbR^{d\times d},
 $$ has an inverse function $\phi$: 
 $$
 q_i = \phi_i\big(t,\mu, \bx,  z, q',  \g, a_0\big).
 $$ 
 \end{assum}
 
 \begin{rem}
 \label{rem-assumvol1}
 Note that $\g$ corresponds to $\pa_{xx} {\d U\over \d\mu}$. Under appropriate technical conditions, one can a prior show that $|\pa_{xx} {\d U\over \d\mu}|\le C_0$ for some constant $C_0>0$. Then in Assumption \ref{assum-vol2} it suffices to assume that $\phi$ exists for $|\g|\le C_0$.
 \end{rem}
 
 Under Assumption \ref{assum-vol2}, \reff{vol-q}  determines uniquely 
 \bea
 \label{vol-phi}
 \pa_{\bx^0} v_i(\L^{\k,\a^0}_t) = \phi_i(\Pi^{\k,\a^0}_t),\q \Pi^{\k,\a^0}_t:=\Big(\L^{\k,\a^0}_t,  Y^{\k,\a^0}_t, \hat Z^{\k,\a^0}_t,  \a^0_t\Big).
 \eea
 Plug this and \reff{vol-uZ} into \reff{vol-SDE} and \reff{vol-BSDE}, we obtain the following FBSDE:
\bea
\label{vol-FBSDE}
\left.\ba{lll}
\dis X^{\k, \a^0, i}_t = x_i + \int_0^t (\pa_{z_i}H)^\phi(\Pi^{\k,\a^0}_s)ds+ \int_0^t  (\pa_{\g_{0i}}H)^\phi(\Pi^{\k,\a^0}_s)dB_s + \k B^i_t,\\
\dis Y^{\k, \a^0, i}_t =  G_i(\L^{\k,\a^0}_T)+ \int_t^T (\pa_{x_i} H)^\phi_s(\Pi^{\k,\a^0}_s) ds  -\int_t^T Z^{\k,\a^0, 0,i}_s dB_s - { \sum_{j=1}^d} \int_t^T \k Z^{\k,\a^0,i, j}_sdB^j_s,\\
\dis \Phi^\phi(\L_t,  y, \hat \b, a_0) = \Phi(\L_t, y, \phi(\L_t, y, \hat\b, a_0), \b,  a_0), \q y\in \dbR^d, \hat\b = (\b^0, \b)\in \dbR^d\times \dbR^{d\times d}.
 \ea\right.
 \eea

\begin{prop}
 \label{prop-vol-supa2}
 Let all the conditions in Proposition \ref{prop-vol-supa} hold such that $U^{\k, \a^0}$ is sufficiently smooth with bounded derivatives, and assume further that Assumption \ref{assum-vol2} hold true. Then the above FBSDE is well-posed, and, for the $F$ in \reff{vol-a*},
 \bea
 \label{Vka0rep}
 V_\k(\a^0) = \dbE\Big[ g(\L^{\k,\a^0}_T) + \int_0^T F^\phi(\Pi^{\k,\a^0}_s, \a^0_s)ds\Big].
 \eea
 \end{prop}
\proof Given the well-posedness of \reff{vol-BPDE}, from the above analysis one can easily prove the well-posedness of FBSDE \reff{vol-FBSDE}. Moreover, applying It\^{o} formula on $U^{\k,\a^0}(\L^{\k,\a^0}_t)$ and recalling \reff{vol-BPDE}, one can easily verify \reff{Vka0rep}.
\qed

To analyze $V_\k = \sup_{\a^0\in \cA} V_\k(\a^0)$, as in the previous section, we rewrite  the last equation of \reff{vol-FBSDE} in a forward form. That is, given $\a^0\in \cA$,  $y\in\dbR^d$, and $\b^0, \b\in \dbL^2(\dbF')$ taking values in $\dbR^d$ and $\dbR^{d\times d}$, respectively.  Denote the index $\th := (\a^0, y, \hat\b)$ with $\hat\b := (\b^0, \b)$, and introduce
\bea
\label{vol-forward}
\left.\ba{lll}
\dis X^{\k,\th,i}_t = x_i + \int_0^t (\pa_{z_i}H)^\phi_s(\Pi^{\k,\th}_s, \a^0_s) ds + \int_0^t (\pa_{\g^0_{i}} H)^\phi_s(\Pi^{\k,\th}_s, \a^0_s) dB_s + \k B^i_t;\\
\dis Y^{\k,\th,i}_t =  y_i- \int_0^t (\pa_{x_i} H)^\phi_s(\Pi^{\k,\th}_s, \a^0_s) ds +\int_0^t \b^{0i}_s dB_s +{ \sum_{j=1}^d} \int_0^t \k \b^{ij}_s dB^j_s,\\
\dis J_\k(\th) := \dbE\Big[ g(\L^{\k,\th}_T)+ \int_{{ 0}}^T F^\phi(\Pi^{\k,\th}_s, \a^0_s)ds \Big],\\
\dis \mbox{where}\q  \L^{\k,\th}_t:=\big( t,\cL_{ \bX^{\k,\th}_{\cd\wedge t}}, \bX^{\k, \th}_{\cd\wedge t}\big),\q \Pi^{\k,\th}_t :=\Big(\L^{\k,\th}_t,  Y^{\k,\th}_t, \hat\b_t, { \a^0_t\Big)}.
 \ea\right.
 \eea
 Then, by Proposition \ref{prop-vol-supa2}, we have
 \bea
 \label{vol-Vk2}
\dis V_\k(\a^0) = \sup_{ y\in \dbR^d, \hat\b\in \dbL^2(\dbF')} J_\k(\a^0, y, \hat\b),\q\mbox{subject to}\q Y^{\k,\th}_T = G(\L^{\k,\th}_T).
 \eea
 
We now introduce the penalization to get rid of the constraint, for $\l>0$ large:
  \bea
 \label{vol-Vkl}
\dis  V_{\l,\k} = \sup_{\a^0\in \cA, y\in \dbR^d, \hat\b\in \dbL^2(\dbF')}  J_{\l, \k}( \a^0, y, \b),~\mbox{where}~ J_{\l,\k}(\th) := J_\k(\th) -\l \dbE\Big[\big|Y^{\k,\th}_T - G(\L^{\k,\th}_T)\big|^2\Big].
 \eea
 
 \begin{assum}
 \label{assum-vol3}
 (i) Assume $H$ is sufficiently smooth and the involved derivatives are bounded;
 
\no (ii) The set $\cA_0$ of $\a^0$ satisfying the requirements in Proposition \ref{prop-vol-supa2} is dense in $\cA$;

\no (iii) For any $(\l, \k)$, \reff{vol-Vkl} has ${1\over \l}$-optimal control $(\a^{0,\l,\k}, y^{\l,\k}, \hat\b^{\l,\k})$ such that $\a^{0,\l,\k}\in \cA_0$, and the corresponding $(u^{\a^{0,\l,\k}}, v^{\a^{0,\l,\k}})\in C^3(\dbF)\times C^2(\dbF)$ and  $U^{\k, \a^{0,\l}}$ have the required derivatives uniformly bounded, uniformly in $(\l,\k)$. 
 \end{assum}
 
 The following result follows the same arguments as in Theorem \ref{thm-drift-VKappalambda}.
  \begin{thm}
 \label{thm-vol-Vkl}
 Let Assumptions \ref{assum-vol1}, \ref{assum-vol2}, and \ref{assum-vol3} hold. Then $\lim_{\l\to\infty} V_{\l,\k} = V_\k$.
 \end{thm}

  Finally we compute $V_{\l,\k}$.  Recall the state space $\dbX^{2d}$ for $(B_\cd, X, Y)$. For any $t_0\in [0, T]$ and $\zeta, \eta\in \dbL^2(\cF'_{t_0}, \dbR^d)$, introduce:
 \bea
 \label{vol-U}
 \left.\ba{lll}
 \dis U_{\l,\k}(t_0, \xi, \eta) :=\sup_{\a^0\in \cA}\sup_{\hat\b\in \dbL^2(\dbF')} \dbE\Big[ g(\L^{\k,\th}_T)- \l\big|Y^{\k,\th}_T - G(\L^{\k,\th}_T)\big|^2+ \int_{t_0}^T F^\phi_s(\Pi^{\k,\th}_s)ds \Big],\\
 \dis \mbox{where}\q \th = (t_0, \zeta, \eta; \a^0, \hat\b),\q \L^{\k,\th}_t :=  \big(t,\cL_{\bX^{\k,\th}_{\cd\wedge t}}, \bX^{\k, \th}_{\cd\wedge t},\big),\q \Pi^{\k, \th}_t := \Big(\L^{\k,\th}_t,  Y^{\k,\th}_t, \hat\b_t,  \a^0_t\Big),\\
 \dis  X^{\k,\th,i}_t = \zeta_i + \int_{t_0}^t (\pa_{z_i}H)^\phi_s(\Pi^{\k,\th}_s) ds + \int_{t_0}^t (\pa_{\g^0_{i}} H)^\phi_s(\Pi^{\k,\th}_s) dB_s + \k (B^i_t-B^i_{t_0});\\
\dis Y^{\k,\th,i}_t =  \eta_i- \int_{t_0}^t (\pa_{x_i} H)^\phi_s(\Pi^{\k,\th}_s) ds +\int_{t_0}^t \b^{0i}_s dB_s + { \sum_{j=1}^d} \int_{t_0}^t \k \b^{ij}_s dB^j_s.
 \ea\right.
 \eea
 { Note that the Brownian motion increments $ B^i_t-B^i_{t_0}$ and $dB^j_s$, $s\ge t_0$, are independent of $(B_{\cd\wedge t_0}, \zeta, \eta)$. Then by \reff{vol-U} 
 one can easily show that $U_{\l, \k}$ is law invariant in terms of $(B_{\cd\wedge t_0}, \zeta, \eta)$ and is adapted. That is,  by abusing the notation we  may denote}
 \bea
 \label{vol-U2}
 U_{\l,\k}(t_0, \mu) := U_{\l,\k}(t_0, \zeta, \eta),\q \mu \in \cP_2(\dbX^{2d}),  \zeta, \eta\in \dbL^2(\cF'_{t_0})~\mbox{such that}~ \cL_{({ B_{\cd\wedge t_0}}, \zeta, \eta)} = \mu.
 \eea
 We emphasize that here we need the joint law  $\cL_{(B_\cd, \zeta, \eta)}$, rathe than  $\cL_{(\zeta, \eta)}$. 
 Moreover, one can easily verify the DPP: for any $t_0<t$,
 \bea
 \label{vol-DPP}
 U_{\l,\k}(t_0, \mu) = \sup_{\a^0\in \cA}\sup_{\hat\b\in \dbL^2(\dbF')} \Big\{ U_{\l, \k}(t, \cL_{(B_\cd, X^{\k, \th}_t, Y^{\k, \th}_t)}) + \dbE\big[\int_{t_0}^t F^\phi_s(\Pi^{\k,\th}_s)ds \big]\Big\}.
 \eea
 Then, applying the It\^{o} formula we have
 \beaa
  \pa_t U_{\l,\k}(t, \mu) + {1\over 2} \dbE\big[\pa_{\bx^0\bx^0}{\d U_{\l,\k}\over \d\mu}(\Xi_t)\big] + \sup_{\a^0\in \cA}\sup_{\hat\b\in \dbL^2(\dbF')} \dbE\big[ \cH(U_{\l,\k}, \Xi_t, \Pi_t)\big]=0,
 \eeaa
 where,  $\Xi_t := (t, \mu, B_\cd, \zeta, \eta)$, $\Pi_t := (t, \cL_{(B_\cd, \zeta)}, B_\cd, \zeta, \eta,\hat \b_t, \a^0_t)$, and
 \bea
 \label{vol-cHk}
\left.\ba{lll}
\dis \cH(U, \Xi_t, \Pi_t) :=\dbE\Big\{
\sum_{i=1}^d\Big[{1\over 2} \pa_{x_ix_i}{\d U\over \d\mu}(\Xi_t) [|(\pa_{\g^0_i} H)^\phi(\Pi_t)|^2 + \k^2] + \pa_{\bx^0 x_i}{\d U\over \d\mu}(\Xi_t) (\pa_{\g^0_i} H)^\phi(\Pi_t)\\
\dis\q  + {1\over 2} \pa_{y_iy_i}{\d U\over \d\mu}(\Xi_t) [|\b^{0i}_t|^2 + \k^2\sum_{k=1}^d |\b^{ik}_t|^2]   +\pa_{\bx^0 y_i}{\d U\over \d\mu}(\Xi_t) \b^{0i}_t\Big]  \\
\dis\q + \sum_{1\le i<j\le d} \Big[\pa_{x_ix_j}{\d U\over \d\mu}(\Xi_t) (\pa_{\g^0_i} H\pa_{\g^0_j} H)^\phi(\Pi_t)  +\pa_{y_iy_j}{\d U\over \d\mu}(\Xi_t)[\b^{0i}_t\b^{0j}_t + \k^2\sum_{k=1}^d \b^{ik}_t\b^{jk}_t]\Big]\\ 
 \dis\q+ \sum_{i, j=1}^d \pa_{x_i y_j} {\d U\over \d\mu}(\Xi_t) \big[(\pa_{\g^0_i} H)^\phi(\Pi_t)\b^{0j}_t +  \k^2 \b^{ji}_t \big]\Big]\\
 \dis\q + \sum_{i=1}^d \Big[\pa_{x_i}{\d U\over \d\mu}(\Xi_t) (\pa_{z_i} H)^\phi(\Pi_t) - \pa_{y_i}{\d U\over \d\mu}(\Xi_t) (\pa_{x_i} H)^\phi(\Pi_t)\Big]  + F^\phi(\Xi_t).
 \ea\right.
 \eea
  Note that $\a^0_t$ is $\cF_t$-measurable, and $\hat\b_t$ as well as $\zeta, \eta$ are $\cF'_t$-measurable. Then, for  deterministic $a_0, \hat\b$, for the same $\Xi_t$ and denoting $\pi_t := (t, \cL_{(B_\cd, \zeta)}, B_\cd, \zeta, \eta,\hat \b, a_0)$, we have 
 \bea
 \label{vol-HJB}
 \left.\ba{c}
\dis  \pa_t U_{\l,\k}(t, \mu) + {1\over 2} \dbE\big[\pa_{\bx^0\bx^0}{\d U_{\l,\k}\over \d\mu}(\Xi_t)\big]+  \dbE\Big[ \sup_{ a_0\in A} \dbE\big[\sup_{\hat\b\in \dbR^{d\times (d+1)}}\cH(U_{\l,\k}, \Xi_t, \pi) \big|\cF_t\big]\Big]=0,\ms\\
\dis U_{\l,\k}(T, \mu) = \dbE\Big[g(\cL_{(B, \zeta)}, B_\cd, \zeta)- \l \big|\eta - G(\cL_{(B, \zeta)}, B_\cd, \zeta)\big|^2\Big].
\ea\right.
 \eea
 
\begin{thm}
\label{thm-vol-leader1}
Let Assumptions \ref{assum-vol1}, \ref{assum-vol2}, and \ref{assum-vol3} hold, and the HJB equation \reff{vol-HJB} on the Wasserstein space has a classical solution $U_{\l,\k}$. Then $V_{\l,\k} = \sup_{y\in \dbR} U(0, \d_{(0, x, y)})$.
\end{thm}
{ We note that, as in Remark  \ref{rem-drift-leader1}, there might be face lifting issue here. In that case Theorem \ref{thm-vol-leader1} remains true if the \reff{vol-HJB} with the correct terminal condition $U_{\l,\k}(T-, \mu)$ has a classical solution.}

\begin{rem}
\label{rem-vol-nonMF}
When $b, \si, f, g$ are independent of $\mu$, then so is $H$ in \reff{vol-H}. In this case, one can easily show that the $U^{\k,\a^0}$ in Proposition \ref{prop-vol-supa} takes the form: $U^{\k, \a^0}(t, \mu) = \int_{\dbX^d} u^{\k, \a^0}(t, \bx) \mu(d\bx)$, namely $u^{\k, \a^0}(t, \bx) := {\d U^{\k, \a^0}\over \d \mu}(t, \mu, \bx)$ is independent of $\mu$. Moreover, 
\beaa
u_t(x) := u^{\k, \a^0}(t, B_\cd, x),\q v_t(x) := \pa_{\bx^0} u^{\k, \a^0}(t, B_\cd, x), 
\eeaa
 satisfy the following BSPDE:
 \bea
 \label{vol-BSPDE}
 \left.\ba{lll}
 \dis u_t(x) = g(B_\cd, x) + \int_t^T H(s, B_\cd, x, \pa_x u_s(x),  \pa_x v_s(x), \pa_{xx} u_s(x),  , \a^0_s)\big] ds - \int_t^T v_s(x)dB_s.
 \ea\right.
 \eea
However, due to the different measurability of $\a^0$ and $\hat\b$, HJB equation \reff{vol-HJB} cannot be reduced to a BSPDE as in Theorem \ref{thm-drift-leader1}.  In particular, ${\d U_{\l,\k}\over \d\mu} (t, \mu, \bx^0, x, y)$ depends on $\mu$.
\end{rem}  
  
\subsection{Application to \reff{Vlambda-2} for the Stackelberg game}
\label{sect-app1}
We can formally apply the results in this section to compute $V_\l$ in \reff{Vlambda-2} for the original Stackelberg game problem. While $\a^0$ remains the same, the $\a$ here corresponds to $(\a,\b)$ in  \reff{Vlambda-2}. Note that the $X$ there is the Brownian motion $B$, see \reff{X}, so we have $d=n+1$ with state processes $(\cX^{\hat\a, \b, 1}, \cds, \cX^{\hat\a, \b, n}, M^{\hat\a}_t)$. For $\mu\in \cP_2(\dbX^{n+1})$ and $x=(x_1,\cds, x_n, x_{n+1})$, the terminal condition takes the form:
\beaa
g(\mu, \bx^0, x) =  x_{n+1} g_0(\bx^0_T)  + \l \sum_{i=1}^n |x_i - g_i(\bx^0_T)|^2 - \l \sum_{i=1}^n \Big(\int_{\dbX^{n+1}} [\tilde x_i - g_i(\tilde \bx^0_T)] \mu(d\tilde \bx)\Big)^2. 
\eeaa
Then we can  easily derive the HJB equation \reff{vol-HJB} for this system. We note that \reff{vol-HJB}  involves additional controls $\hat\b$ and additional parameters $\l$ (and $\k$), which are different from the $\b$ and $\l$ already involved in \reff{Vlambda-2}.

We shall emphasize though that this derivation is just heuristic. The results in this section require very strong technical conditions, which are unlikely satisfied by the system for \reff{Vlambda-2}. In particular, the $\b$ in \reff{cXi} is unbounded, implying that the $\si$ in \reff{vol-V0} is unbounded. So some serious efforts, including possibly additional approximations, are needed in order to have a rigorous solution to \reff{Vlambda-2}. In this paper we content ourselves with the analysis of the structure of the problem, and we leave the technical challenges to future research.
 
\subsection{Application to \reff{PA-Vn-2} for the principal-agent problem}
\label{sect-app2}

We now turn to \reff{PA-Vn-2}, again with heuristic analysis only. Since both \reff{PA-Vn} and \reff{vol-Vkl} involve a parameter $\l$, we use $\l_1$ for \reff{PA-Vn} and $\l_2$ for \reff{vol-Vkl}. Recall \reff{PA-Vn-2}. For fixed  $\xi\in \dbL^2(\cF_T)$, denote
\bea
\label{PA-Vn-3}
\left.\ba{lll}
\dis V_{\l_1}(R,\xi) := \sup_{\a^0\in \cA}\inf_{(\a, \b)\in \cA^n\times (\dbL^2(\dbF))^n}\Big\{  \l \sum_{i=1}^n \phi\big(R_i, \dbE[g_i(\xi_i)-\cX^{\hat\a, \b, i}_T]\big)  \\
\dis \q + \dbE\Big[M^{\hat\a}_T g_0(X_T-\sum_{i=1}^n \xi_i)+\l \sum_{i=1}^n \big|\cX^{\hat\a, \b, i}_T - g_i(\xi_i)\big|^2  \\
\dis \q + \int_0^T \big[M^{\hat\a}_s f_0(s, X_s,\hat \a_s) ds+ \l \sum_{i=1}^n  |\ol F_i(s, X_s, \hat\a_s, \b^{ i}_s) - F_i(s, X_s, \hat\a_s, \b^{i}_s)|^{3\over 2}\big]ds \Big]\Big\}.
\ea\right.
\eea
As in Subsection  \ref{sect-app1}, this is problem \reff{vol-V0} with $d=n+1$, with appropriate coefficients $b, \si, f$ and terminal condition, for $\mu\in \cP_2(\dbX^{n+1})$ and $\bx = (\bx^0, x_1, \cds, x_n, x_{n+1})$, 
\bea
\label{gxi}
\left.\ba{c}
\dis g^{\l_1,\xi}(\mu, \bx) = x_{n+1} g_0(\bx^0_T - \sum_{i=1}^n \xi_i(\bx^0_\cd)) + \l_1\sum_{i=1}^n |x_i - g_i(\xi_i(\bx^0_\cd))|^2 \\
\dis\qq + \l_1 \sum_{i=1}^n  \phi\Big(R_i, \int_{\dbX^{n+1}} [g_i(\xi_i(\tilde\bx^0_\cd))- \tilde x_i] \mu(d\tilde \bx)\Big).
\ea\right.
\eea 
Recall \reff{vol-Y}. This implies that, denoting by $\phi'$ the derivative of $\phi$ with respect to its second argument, 
\bea
\label{Gxi}
\left.\ba{c}
\dis G^{\l_1,\xi}_i(\mu, \bx) = 2 \l_1 [x_i - g_i(\xi_i(\bx^0_\cd)] - \l_1 \phi'\Big(R_i, \int_{\dbX^{n+1}} [g_i(\xi_i(\tilde\bx^0_\cd))- \tilde x_i] \mu(d\tilde \bx)\Big), ~ i=1,\cds, n;\\
\dis G^{\l_1,\xi}_{n+1}(\mu, \bx) = g_0(\bx^0_T - \sum_{i=1}^n \xi_i(\bx^0_\cd)).
\ea\right.
\eea 

From the results of this section, $V_{\l_1}(\xi)$ can be approximated by 
\bea
\label{Vlkxi}
V_{\l_1, \l_2, \k}(\xi) = \sup_{y\in \dbR^{n+1}}U_{\l_2, \k}^{\l_1, \xi}\big(0, \d_{(0, x^0, y)}\big),\q x^0 := (0,\cds, 0, 1),
\eea
 where $U_{\l_2, \k}^{\l_1, \xi}(t, \mu)$, $(t, \mu) \in [0, T]\times \cP_2(\dbX^{2(n+1)})$ satisfies the HJB equation \reff{vol-HJB} with appropriate $\cH$ and the following terminal condition: recalling \reff{vol-HJB}, \reff{gxi}, and \reff{Gxi}, and letting $\mu_1= \cL_{(B_\cd, \zeta)}$ denote the first marginal distribution of $\mu = \cL_{(B_\cd, \zeta, \eta)}$,  \bea
 \label{Ulkxi-terminal}
 U_{\l_2, \k}^{\l_1, \xi}(T, \mu) = \cG_{\l_2, \k}^{\l_1, \xi}(\mu) := \dbE\Big[g^{\l_1,\xi}(\mu_1, B_\cd, \zeta) - \l_2  |\eta -  G^{\l_1,\xi}(\mu_1, B_\cd, \zeta)|^2\Big].
 \eea
 Note that $V_{\l_1} = \sup_{\xi\in \dbL^2(\cF_T)} V_{\l_1}(\xi)$, which can be approximated by 
 \bea
 \label{olVlk}
 \left.\ba{c}
 \dis V^{\l_1}_{\l_2,\k} := \sup_{\xi\in \dbL^2(\cF_T)} V^{\l_1}_{\l_2, \k}(\xi) = \sup_{y\in \dbR^{n+1}} \ol U_{\l_2, \k}^{\l_1}\big(0, \d_{(0, x^0, y)}\big),\q  \ol U_{\l_2, \k}^{\l_1} := \sup_{\xi\in \dbL^2(\cF_T)} U_{\l_2, \k}^{\l_1, \xi}.
 \ea\right.
 \eea
 Denote
 \bea
 \label{olcG}
\dis \ol\cG_{\l_2, \k}^{\l_1}(\mu) := \sup_{\xi\in \dbL^2(\cF_T)}\cG_{\l_2, \k}^{\l_1, \xi}(\mu)  = \sup_{\xi\in \dbL^2(\cF_T)} \dbE\Big[g^{\l_1,\xi}(\mu_1, B_\cd, \zeta) - \l_2  |\eta -  G^{\l_1,\xi}(\mu_1, B_\cd, \zeta)|^2\Big].
\eea
Note that the Hamiltonian $\cH$ for the equation of $U_{\l_2, \k}^{\l_1, \xi}(t, \mu)$ is independent of $\xi$. Then, by the comparison principle of HJB equations on Wasserstein space (cf. \cite{ZTZ}), { at least formally,}  $\ol U_{\l_2, \k}^{\l_1}$ satisfies HJB equation \reff{vol-HJB} with the same $\cH$ and terminal condition $\ol\cG_{\l_2, \k}^{\l_1}$. That is, by solving this HJB equation, we { expect} $\dis V^{\l_1}_{\l_2,\k} = \sup_{y\in \dbR^{n+1}} \ol U_{\l_2, \k}^{\l_1}\big(0, \d_{(0, x^0, y)}\big)$, which is an approximate value of $V^{\l_1}$. 

\begin{rem}
\label{rem-optimalxi}
We note again that we are ignoring the strong technical conditions. However, the constructions $\phi$ in \reff{PA-miny},  $g^{\l_1,\xi}$ in \reff{gxi}, $G^{\l_1, \xi}$ in \reff{Gxi}, $\cG_{\l_2, \k}^{\l_1, \xi}$ in \reff{Ulkxi-terminal} are all explicit, with essentially no technical conditions, and then \reff{olcG} provides us { a natural way to construct a maximizer for $\xi$ (assuming its existence)}. We conjecture that this  constructed $\xi$ would indeed be an approximate optimal contract for much more general principal-agent problems violating our technical requirements.  
\end{rem}

\section{Some examples}
\label{sect-eg}
\setcounter{equation}{0}

In this section we present a few examples to illustrate some points mentioned in the paper. The first example concerns the difference between \reff{leader-sup} and  \reff{leader-inf}. For deterministic one period games, we use $\hat a = (a_0, a)=(a_0, a_1, \cds, a_n)$ to denote the controls taking values in the control set $A$.
\begin{eg}
\label{eg-supsup}
Consider a deterministic one period zero sum game with $n=2$:
\bea
\label{eg1-zerosum}
J_0 = - [J_1 + J_2].
\eea
  All players have action space $A=\{0,1\}$, and the payoff $J(\hat a) = (J_1(\hat a), J_2(\hat a))$ is as follows:
%

\begin{table}[h]
\begin{center}
{\begin{tikzpicture}[scale=0.45]
    \draw[thick] (0,0) rectangle (12,7);

    \draw[thick] (3.2,0) -- (3.2,7);
    \draw[thick] (7.6,0) -- (7.6,7);

    \draw[thick] (0,2.3) -- (12,2.3);
    \draw[thick] (0,4.6) -- (12,4.6);

    \draw[thick] (0,7) -- (1.8,4.6);
    \draw[thick] (0,7) -- (3.2,5.8);

    \node at (0.8,5.2) {$a_1$};
    \node at (2.45,6.6) {$a_2$};
    \node at (2.25,5.5) {$J(\hat a)$};

    \node at (5.4,5.8) {$0$};
    \node at (9.8,5.8) {$1$};

    \node at (1.6,3.45) {$0$};
    \node at (1.6,1.15) {$0$};

    \node at (5.4,3.45) {$(2,2)$};
    \node at (9.8,3.45) {$(0,0)$};
    \node at (5.4,1.15) {$(0,0)$};
    \node at (9.8,1.15) {$(3,3)$};
\end{tikzpicture} 
\qq
\begin{tikzpicture}[scale=0.45]
    \draw[thick] (0,0) rectangle (12,7);

    \draw[thick] (3.2,0) -- (3.2,7);
    \draw[thick] (7.6,0) -- (7.6,7);

    \draw[thick] (0,2.3) -- (12,2.3);
    \draw[thick] (0,4.6) -- (12,4.6);

    \draw[thick] (0,7) -- (1.8,4.6);
    \draw[thick] (0,7) -- (3.2,5.8);

    \node at (0.8,5.2) {$a_1$};
    \node at (2.45,6.6) {$a_2$};
    \node at (2.25,5.5) {$J(\hat a)$};

    \node at (5.4,5.8) {$0$};
    \node at (9.8,5.8) {$1$};

    \node at (1.6,3.45) {$0$};
    \node at (1.6,1.15) {$0$};

    \node at (5.4,3.45) {$(1,1)$};
    \node at (9.8,3.45) {$(0,0)$};
    \node at (5.4,1.15) {$(0,0)$};
    \node at (9.8,1.15) {$(100,100)$};
\end{tikzpicture}}
  \caption{Left: $a_0=0$,\q Right: $a_0=1$}
\end{center}
\end{table}

\no Clearly, for both $a_0=0$ and $a_0=1$, the followers have two equilibria: 
$$\cE(0) = \cE(1) = \{(0,0), (1,1)\}.
$$
Recall \reff{eg1-zerosum}, one can easily compute that:
\beaa
&&\dis \ol v_0(0) := \sup_{a^*\in \cE(0)} J_0(0, a^*) = \max\big(J_0(0,0,0), J_0(0,1,1)\big)=\max(-4, -6) = -4;\\
&&\dis \ol v_0(1) := \sup_{a^*\in \cE(1)} J_0(1, a^*) = \max\big(J_0(1,0,0), J_0(1,1,1)\big)=\max(-2, -200) = -2;\\
&&\dis \ol v_0 = \max\big(\ol v_0(0), \ol v_0(1)\big) = \max(-4, -2) = -2,\q\mbox{with $a_0^* = 1$};\\
&&\dis \ul v_0(0) := \inf_{a^*\in \cE(0)} J_0(0, a^*) = \min\big(J_0(0,0,0), J_0(0,1,1)\big)=\min(-4, -6) = -6;\\
&&\dis \ul v_0(1) := \inf_{a^*\in \cE(1)} J_0(1, a^*) = \min\big(J_0(1,0,0), J_0(1,1,1)\big)=\min(-2, -200) = -200;\\
&&\dis \ul v_0 = \max\big(\ul v_0(0), \ul v_0(1)\big) = \max(-6, -200) = -6,\q\mbox{with $\a_0^* = 0$}.
\eeaa
 If we go with the problem \reff{leader-sup}, then the leader will choose $a_0 = 1$, by expecting the followers will select the equilibrium $a = (0,0)$. However, in this case $J(1, 0, 0) = (1,1)$ and $J(1,1,1)= (100, 100)$, obviously the followers would select the equilibrium $a^*=(1,1)$, and then the leader will actually end up with utility $J_0(1, 1,1) = -200$, which is much worse than $\ul v_0 = -6$, the value the leader can lock in if she chooses $a_0=0$. 
\end{eg}

The next result concerns the difference between \reff{leader-inf} and \reff{leader}.

\begin{eg}
\label{eg-Nashe}
Again consider a deterministic one period game with one leader and two followers. Let $A=[0, \infty)$, and for $\hat a=(a_0, a_1, a_2)\in A^3$,
\beaa
J_1(\hat a) = J_2(\hat a) = {a_1 a_2\over (1+a_1)(1+a_2)},\q J_0(\hat a) = -|a_0|^2 - 2\phi\big({a_1a_2\over (1+a_1)(1+a_2)}\big) a_0 + 4a_0,
\eeaa
where $\phi(x) := x\1_{[0, 1]}(x) + {1\over x}\1_{(1,\infty)}$, which is increasing on $[0,1]$ and decreasing on $[1,\infty)$.   Then
\bea
\label{Nashe-inequality}
V_0  \le 1 < 4 = \ul v_0.
\eea
\end{eg}

\proof First, one can easily see that $\cE(a_0) = \{(0,0)\}$ for any $a_0\in A$. Then
\beaa
\ul v_0 = \sup_{a_0\in A} J_0(a_0, 0, 0) = \sup_{a_0\in A} (-|a_0|^2 + 4a_0) = 4.
\eeaa

On the other hand, for any $\e>0$ small and any $a_0\in A$, we can easily verify that 
\beaa
(a^\e_1, a^\e_2)\in \cE_\e(a_0), \q\mbox{where}\q a^\e_1=a^\e_2 := {1-\e\over \e}.
\eeaa
Then
\beaa
&&\sup_{ a\in \cE_\e(a_0)} \phi\big({a_1a_2\over (1+a_1)(1+a_2)}\big) \ge \phi\Big({a^\e_1a^\e_2\over (1+a^\e_1)(1+a^\e_2)}\Big) = {  (1-\e)^2};\\
&&\dis v_\e(a_0) = \inf_{ a\in \cE_\e(a_0)} J_0(a_0, a) \le -|a_0|^2 -  2a_0 { (1-\e)^2} + 4a_0;\\
&&\dis V_\e := \sup_{a_0\in A} v_\e(a_0)  \le \sup_{a_0\ge 0} \big[ -|a_0|^2 -  2a_0 { (1-\e)^2} + 4a_0\big] = \big(2-{ (1-\e)^2}\big)^2.
\eeaa
This implies \reff{Nashe-inequality} immediately.
\qed

We now present an example in the setting of Section \ref{sect-stackelberg} which we can solve explicitly.
\begin{eg}
\label{eg-explicit}
 Set $n=2$, $A=[0,2]$, and 
 \bea
 \label{explicit-coefficients}
 \left.\ba{c}
  \dis b\equiv 0, \q  \sigma \equiv 1, \q g_i(x) := x^2, ~ i=0,1,2;\\
 \dis f_i(t,x,{\hat{a}})=-|a_1-a_2|^{4\over 3},  i=1,2;\q   f_0(t,x,\hat{a})=a_0 +|a_1|^2+|a_2|^2 -2a_1 a_0.
 \ea\right.
 \eea
 Then
 \beaa
 V_0 = {3\over 2} T,\q V_\l = {3\over 2}T -{4\l + 3\over 8(4\l^2 + 3\l +1)}T.
 \eeaa
 In particular, we have $\lim_{\l\to \infty} V_\l = V_0$.
\end{eg}
\proof By \reff{explicit-coefficients}, it is clear that $\dbP^{\hat \a} = \dbP$ and 
\beaa
F_i(t, x, \hat a, z_i)= - |a_1-a_2|^{4\over 3},\q  \ol F_i(t, x, \hat a, z_i) \equiv 0.
\eeaa
One can easily show that $\a\in \cE(\a^0)$ if and only if $\a^1=\a^2$, and for $\e>0$,  $\a\in \cE_\e(\a^0)$ if and only if $\dbE\big[\int_0^T |\a^1_t - \a^2_t|^{4\over 3}dt\big] \le \e$. Note that
\beaa
v_\e(\a^0) = \inf_{\a\in \cE_\e(\a^0)} \dbE\Big[|B_T|^2+\int_0^T \big[\a^0_t +|\a^1_t|^2+|\a^2_t|^2-2\a^1_t\a^0_t \big]dt\Big],
\eeaa
and that, since $|\a^1|, |\a^2|\le 2$,
\beaa
&&\dis \Big|\int_0^T |\a^2_t|^2 dt - \int_0^T |\a^1_t|^2 dt\Big| \le \sqrt{\e}\int_0^T |\a^1_t|^2 dt + (1+{1\over \sqrt{\e}}) \int_0^T (\a^2_t-\a^1_t)^2 dt \\
&&\dis \le  \sqrt{\e}\int_0^T |\a^1_t|^2 dt  +c(\e+\sqrt{\e}),\q\mbox{where}\q c:= 2^{2\over 3}.
\eeaa
Then, one can easily see that 
\bea
\label{explicit-est1}
v_\e(\a^0) \le v_\e^{1+\sqrt{\e}}(\a^0) +  c(\e+\sqrt{\e}),\q v_\e(\a^0) \ge v_\e^{1-\sqrt{\e}}(\a^0) -  c(\e+\sqrt{\e}),
\eea
where, for $0<\th <2$, 
\beaa
v^\th_\e(\a^0) &:=&  \inf_{\a^1\in \cA} \dbE\Big[|B_T|^2+\int_0^T \big[\a^0_t +|\a^1_t|^2+\th |\a^1_t|^2-2\a^1_t\a^0_t  \big]dt\Big]\\
&=& T + \inf_{\a^1\in \cA} \dbE\Big[\int_0^T \big[(1+\th)|\a^1_t|^2 - 2 \a^0_t \a^1_t + \alpha^0_t \big]dt\Big]\\
&=& T + \dbE\Big[\int_0^T \big[\alpha^0_t - {1\over 1+\th}|\a^0_t|^2  \big]dt\Big],
\eeaa
where $\a^{1*}_t = {\a^0_t\over 1+\th}\in A$ whenever $\a^0_t\in A$. 
Denote
\beaa
V_\e^\th := \sup_{\a^0\in \cA} v^\th_\e(\a^0)= T + \sup_{\a^0\in \cA}\dbE\Big[\int_0^T \big[\alpha^0_t  - {1\over 1+\th}|\a^0_t|^2\big]dt\Big] = {5+\th\over 4}T,
\eeaa
where $\a^{0*}_t = {1+\th\over 2}\in A$.
Then, by \reff{explicit-est1}, we have
\bea
\label{explicit-est2}
V_\e \le  {6+\sqrt{\e}\over 4}T+  c(\e+\sqrt{\e}),\q V_\e \ge {6-\sqrt{\e}\over 4}T -  c(\e+\sqrt{\e}),
\eea
This clearly implies that $V_0=\lim_{\e\to 0} V_\e = {3\over 2} T$.

It remains to compute  $V_{\lambda}$.  Note that 
\beaa
\cX^{\hat\a, \b, i}_t =  \int_0^t  \b^{i}_sdB_s,\q
 I(\hat\a, y,\b) := \sum_{i=1}^2 \dbE\Big[ \int_0^T |\a_s^1-\a_s^2|^2 ds + \big|y_i + \int_0^T  \b^{i}_sdB_s - |B_T|^2\big|^2\Big]. 
\eeaa
Then
\beaa
 &&\dis v_\l(\a^0) := \inf_{\a\in \cA^2, y\in \dbR^2, \b\in (\dbL^2(\dbF))^2}  \Big[J_0(\a^0, \a) + \l  I(\a^0, \a, y, \b)\Big]\\
 &&\dis = \inf_{\a^1,\a^2\in \cA, y\in \dbR^2, \b\in (\dbL^2(\dbF))^2}\dbE\Big[|B_T|^2+\int_0^T \big[\a^0_t +|\a^1_t|^2+ |\a^2_t|^2-2\a^1_t\a^0_t  +2\lambda |\a_t^1-\a_t^2|^2\big]dt\\
 &&\qq\qq +\lambda \sum_{i=1}^2 \big|y_i + \int_0^T  \b^{i}_tdB_t -|B_T|^2\big|^2\Big]\\
 &&\dis =T+ \dbE\Big[\int_0^T [\a^0_t- {8\l^2+6\l +2\over (1+4\l)^2}|\a^0_t|^2 ] dt\Big],
\eeaa
where $a^{1*}_t = {1+2\l\over 1+4\l} \a^0_t \in A$, $\a^{2*}_t = {2\l\over 1+4\l} \a^0_t \in A$, whenever $\a^0_t\in A$. Then
\beaa
V_\l := \sup_{\a^0\in \cA} v_\l(\a^0) =  T + {(1+4\l)^2\over 8(4\l^2 + 3\l +1)}T = {3\over 2}T -{4\l + 3\over 8(4\l^2 + 3\l +1)}T,
\eeaa
where $\a^{0*}_t = {(1+4\l)^2 \over 4(4\l^2 + 3\l +1)} \in A$.
\qed

Finally we present a deterministic zero sum game problem with open loop controls, following the idea of \cite[Example 2.1, Remark 2.2]{PTZ}. We shall let $X$ be two dimensional, which does not change the story of Sections \ref{sect-drift} and \ref{sect-vol}. 
\begin{eg}
\label{eg-zerosum}
 Set $A=[0,1]$, and for $\hat\a=(\a^0, \a)$, 
 \bea
 \label{zerosum-model}
 X^{\a^0, 1}_t = \int_0^t \a^0_s ds,~ X^{\a,2}_t = \int_0^t \a_s ds, \q V_0 := \sup_{\a^0\in \cA} \inf_{\a\in \cA} J(\a^0, \a),\q J(\hat\a) := -|X^{\a^0, 1}_T-X^{\a,2}_T|^2.
 \eea
 Then the problem is time inconsistent, in particular, it is different from the solution to the Bellman-Isaacs equation. 
\end{eg}
\proof We first compute $V_0$. Given $\a^0$, noting that  $\{X^{\a,2}_T, \a\in \cA\} = [0, T]$, we have
\beaa
V_0(\a^0) = { \inf_{\a\in \cA} \Big[- \big|X^{\a,2}_T- X^{\a^0, 1}_T\big|^2\Big] }= - \Big(X^{\a^0, 1}_T \vee (T-X^{\a^0, 1}_T)\Big)^2,
\eeaa
Thus
\beaa
V_0 = -\inf_{\a^0\in \cA}  \Big(X^{\a^0, 1}_T \vee (T-X^{\a^0, 1}_T)\Big)^2 = -{T^2\over 4},
\eeaa
with optimal $\a^{0*}_t \equiv  {1\over 2}$. Moreover, $V_0(\a^{0*})$ has two optimal arguments: $\a^*\equiv 0$ or $\a^*\equiv 1$.

We next compute $\ol V_0 := \inf_{\a^0\in \cA}\sup_{\a\in \cA} J(\a^0, \a)$. In this case,
\beaa
{ \ol V_0(\a^0) := \sup_{\a\in \cA} J(\a^0, \a)  = \sup_{\a\in \cA} \Big[- \big|X^{\a,2}_T- X^{\a^0, 1}_T\big|^2\Big] } =0,
\eeaa 
with optimal argument $\a^* = \a^0$. Then clearly $\ol V_0 = 0 > V_0$, and any $\a^0$ is optimal.

Moreover, in this case the Isaacs condition holds, with Hamiltonian:
\beaa
H(z_1, z_2) = \sup_{a_0\in A}\inf_{a\in A} [a_0 z_1 + a z_2] =\inf_{a\in A} \sup_{a_0\in A} [a_0 z_1 + a z_2] = z_1^+ - z_2^-.
\eeaa
Then the Bellman-Isaacs equation is:
\beaa
\pa_t u(t, x_1, x_2) + (\pa_{x_1}u)^+ - (\pa_{x_2} u)^- =0,\q u(T,x_1, x_2) = - (x_1-x_2)^2.
\eeaa
One can easily show that $u(t, x_1, x_2) = -(x_1-x_2)^2$ is the unique classical solution to the PDE. In particular, $V_0 = -{T^2\over 4} \neq 0 = u(0, 0, 0)$. 

We finally explain that the natural dynamic extension of the problem \reff{zerosum-model} is time inconsistent. Indeed, fix $t_0:= {2T\over 3}$. Following the optimal control $\a^{0^*}_t = {1\over 2}$ and $\a^*_t=0$ on $[0, t_0]$, we get $X^1_{t_0} = {T\over 3}$ and $X^2_{t_0} = 0$. Now consider the following problem:
\bea
 \label{zerosum-model2}
 \left.\ba{c}
\dis X^{t_0, \a^0, 1}_t = {T\over 3} + \int_{t_0}^t \a^0_s ds,\q X^{t_0, \a,2}_t = \int_{t_0}^t \a_s ds, \\
\dis V_0(t_0) := \sup_{\a^0\in \cA} \inf_{\a\in \cA} J(t_0, \a^0, \a),\q J(t_0, \hat\a) := -|X^{t_0,\a^0, 1}_T-X^{t_0,\a,2}_T|^2.
\ea\right.
 \eea
 Given $\a^0$, since  $\{X^{t_0, \a,2}_T, \a\in \cA\} = [0, {T\over 3}]$ and $X^{t_0, \a^0, 1}_T\ge {T\over 3}$, we have
\beaa
V_0(t_0, \a^0) := - \sup_{\a\in \cA} \Big|X^{t_0,\a,2}_T- X^{t_0, \a^0, 1}_T\Big|^2 = - \Big(X^{t_0, \a^0, 1}_T \vee |{T\over 3}-X^{t_0,\a^0, 1}_T|\Big)^2 = -|X^{t_0, \a^0, 1}_T|^2.
\eeaa
Then
\beaa
V_0 = -\inf_{\a^0\in \cA} |X^{t_0,\a^0, 1}_T |^2 = -({T\over 3})^2,
\eeaa
with optimal control $\a^{t_0, 0*} \equiv  0$. This is different from $\a^{0*}\equiv {1\over 2}$,  and thus the problem is time inconsistent.

We note that the problem is still time inconsistent if we follow the other optimal control $\a^{0^*}_t = {1\over 2}$ and $\a^*_t=1$ on $[0, t_0]$, we get $X^1_{t_0} = {T\over 3}$ and $X^2_{t_0} = {2T\over 3}$. Now by using the same notations and  consider the following problem:
\bea
 \label{zerosum-model3}
 \left.\ba{c}
\dis X^{t_0, \a^0, 1}_t = {T\over 3} + \int_{t_0}^t \a^0_s ds,\q X^{t_0, \a,2}_t = {2T\over 3}+\int_{t_0}^t \a_s ds, \\
\dis V_0(t_0) := \sup_{\a^0\in \cA} \inf_{\a\in \cA} J(t_0, \a^0, \a),\q J(t_0, \hat\a) := -|X^{t_0,\a^0, 1}_T-X^{t_0,\a,2}_T|^2.
\ea\right.
 \eea
 Given $\a^0$, since  $\{X^{t_0, \a,2}_T, \a\in \cA\} = [{2T\over 3},T]$ and $X^{t_0, \a^0, 1}_T\le {2T\over 3}$, we have
\beaa
V_0(t_0, \a^0) := - \sup_{\a\in \cA} \Big|X^{t_0,\a,2}_T- X^{t_0, \a^0, 1}_T\Big|^2 = -  \big(T-X^{t_0,\a^0, 1}_T\big)^2.
\eeaa
Then
\beaa
V_0 = -\inf_{\a^0\in \cA}  \big(T-X^{t_0,\a^0, 1}_T\big)^2 = -({T\over 3})^2,
\eeaa
with optimal control $\a^{t_0, 0*} \equiv  1$, which is also different from $\a^{0*}\equiv {1\over 2}$.
\qed

\end{document}